\documentclass{amsart}

\usepackage{amsmath,amsthm,amsfonts,amssymb,cmap,enumerate,rotating,graphicx, mathrsfs,mathtools,tikz}
\usetikzlibrary{matrix,arrows}
\usepackage{enumitem}

\usepackage{hyperref}
\hypersetup{
    colorlinks,
    citecolor=black,
    filecolor=black,
    linkcolor=black,
    urlcolor=black
}

\usepackage[textwidth=2.8cm]{todonotes}
\usepackage{cmap}

\newtheorem{thm}{Theorem}[section]
\newtheorem{prop}[thm]{Proposition}

\newtheorem*{thm*}{Theorem}
\newtheorem*{prop*}{Proposition}
\theoremstyle{definition}
\newtheorem{defn}[thm]{Definition}
\newtheorem{rough defn}[thm]{Rough definition}
\newtheorem{eg}[thm]{Example}

\newtheorem{key construction}[thm]{Key construction}
\newtheorem{conjectural definition}[thm]{Conjectural definition}

\newtheorem{rmk}[thm]{Remark}
\newtheorem{observation}[thm]{Observation}

\newcommand*{\defeq}{\mathrel{\vcenter{\baselineskip0.5ex \lineskiplimit0pt
                     \hbox{\scriptsize.}\hbox{\scriptsize.}}}%
                     =}
\newcommand{\series}[1]{[\![t_1, \ldots, t_{#1}]\!]}

\newcommand\EquivTo{\xrightarrow{
   \,\smash{\raisebox{-0.65ex}{\ensuremath{\scriptstyle\sim}}}\,}}
\newcommand{\xrightarrowdbl}[1]{\xrightarrow{#1}\mathrel{\mkern-14mu}\rightarrow
}

\newcommand{\BBA}{\mathbb{A}}
\newcommand{\BBN}{\mathbb{N}}
   
\newcommand{\CA}{\mathcal{A}}
\newcommand{\CB}{\mathcal{B}}

\newcommand{\CD}{\mathcal{D}}

\newcommand{\CH}{\mathcal{H}}

\newcommand{\CK}{\mathcal{K}}
\newcommand{\CL}{\mathcal{L}}

\newcommand{\CP}{\mathcal{P}}

\newcommand{\CY}{\mathcal{Y}}
\newcommand{\CZ}{\mathcal{Z}}

\newcommand{\colim}{\operatornamewithlimits{colim}}
\newcommand{\emb}{\hookrightarrow}
\newcommand{\fSet}{\text{fSet}}
\newcommand{\id}{\operatorname{id}}
\newcommand{\op}{\text{op}}

\newcommand{\surj}{\twoheadrightarrow}

\newcommand{\Ran}{\operatorname{Ran}}
\newcommand{\ChAlg}[1]{\operatorname{ChAlg}(#1)}
\newcommand{\ch}{\mathit{ch}}
\newcommand{\FAlg}[1]{\operatorname{FAlg}(#1)}
\newcommand{\WFAlg}[1]{\operatorname{WFAlg}(#1)}
\newcommand{\FSp}[1]{\operatorname{FSp}({#1})}
\newcommand{\WFSp}[1]{\operatorname{WFSp}({#1})}
\newcommand{\Glue}{\operatorname{Glue}}
\newcommand{\Weak}{\operatorname{Weak}}

\newcommand{\GrGXI}[1]{\operatorname{Gr}_{G, #1^I}}

\newcommand{\pt}{\text{pt}}
\newcommand{\Spec}{\operatorname{Spec}}
\newcommand{\Sch}{\operatorname{Sch}}
\newcommand{\graph}[1]{\{ #1\}}
\newcommand{\Hilb}{\operatorname{Hilb}}
\newcommand{\Supp}{\operatorname{Supp}}
\newcommand{\red}{\mathit{red}}
\newcommand{\redEmb}[1]{\iota_{#1}}

\newcommand{\HilbI}[2]{\CH\!\textit{ilb}_{#1^{#2}}}

\begin{document}
\title[Universal factorization spaces and algebras]{Universal factorization spaces and algebras}
\author[E.~Cliff]{Emily Cliff}

%

\begin{abstract}
We introduce categories of weak factorization algebras and factorization spaces, and prove that they are equivalent to the categories of ordinary factorization algebras and spaces, respectively. This allows us to define the pullback of a factorization algebra or space by an \'etale morphism of schemes, and hence to define the notion of a universal factorization space or algebra. This provides a generalization to higher dimensions and to non-linear settings of the notion of a vertex algebra. 
\end{abstract}

\keywords{Factorization algebra. Factorization space. Vertex algebra. Chiral algebra.}

\subjclass{17B69. 14D99.}

\maketitle

\tableofcontents

\section{Introduction}
\label{sec: intro}
A vertex algebra describes the symmetries of a two-dimensional conformal field theory (CFT), while a factorization space or factorization algebra over a complex curve consists of local data in such a field theory. Roughly, the factorization structure encodes collisions between local operators. A \emph{quasi-conformal} vertex algebra---that is, one equipped with a one-dimensional infinitesimal translation corresponding to Virasoro symmetry of the CFT---gives rise to a chiral algebra on each smooth complex curve, in a way compatible with pullback along \'etale morphisms between curves (see \cite{FBZ2}, and also \cite{HL}). Since it is known from the work of Beilinson--Drinfeld \cite{BD2} and Francis--Gaitsgory \cite{FG} that a chiral algebra over a variety is equivalent to a factorization algebra over the same variety, we thus obtain a family of factorization algebras over all smooth curves, expected to satisfy some compatibility condition related to \'etale morphisms.
\vskip 3mm
In this paper, we delineate this compatibility condition, thereby formulating the definition of a \emph{universal factorization algebra} of any dimension $d$, which is new even in complex dimension one (i.e. for 2D CFTs). We also treat the non-linear analogue, a universal factorization space. A key ingredient is the notion of a weak factorization algebra or space, which we prove is equivalent to an ordinary factorization algebra or space, respectively, and which allows us to define pull-back along \'etale morphisms. As a consequence of these definitions, we put on a firm foundation important work of Kapranov--Vasserot \cite{KV, KV2, KV3, KV4}.  Our definitions and results accommodate all standard examples, are compatible with the already established notions of a universal chiral algebra and a quasi-conformal vertex algebra, and provide a framework for the exploration of vertex-algebra--like symmetries in higher dimensions.
\vskip 3mm

\begin{thm*}[Proposition \ref{prop: comparing chiral and factorization pullback}]
Let $X$ and $Y$ be varieties, and let $\phi: X \to Y$ be an \'etale morphism between them. Then if $\mathcal{A}$ is a factorization algebra over $Y$, the factorization algebra $\phi^*\mathcal{A}$ on $X$ obtained by pulling back along $\phi$ corresponds to the chiral algebra on $X$ obtained by taking taking the pull-back of the chiral algebra on $Y$ corresponding to $\mathcal{A}$. 

It follows that the category of universal chiral algebras of a given dimension $d$ is equivalent to the category of universal factorization algebras of dimension $d$ as defined in this paper. In particular, the category of universal factorization algebras of dimension one is equivalent to the category of quasi-conformal vertex algebras. 
\end{thm*}

The property of universality for a family of factorization spaces or algebras is important for three reasons: first, it is not clear how to extend the definition of a vertex algebra to a higher-dimensional analogue, encoding symmetries of higher-dimensional CFT (see for example \cite{B2}). On the other hand, Francis and Gaitsgory \cite{FG} showed that the definition of a factorization algebra or space extends in a natural way from curves to complex varieties of arbitrary dimension, and so the categories of universal factorization algebras and spaces provide generalizations to higher dimensions and to non-linear settings of the notion of a quasi-conformal vertex algebra. 

Second, even in the one-dimensional setting, there are advantages to working with factorization structures rather than vertex algebras: namely, it facilitates the transfer of information and ideas among algebraic geometry, representation theory, and physics. Examples of factorization spaces are constructed from some moduli spaces arising in geometric representation theory and algebraic geometry; linearizing these factorization spaces, for instance by taking their cohomology with values in different sheaves, provides examples of factorization algebras, which can then be studied using the geometry of the underlying spaces. Conversely, the existence and properties of the factorization structures can be used to better understand the original moduli spaces. 

Finally, universality of a family of factorization spaces or factorization algebras can allow us to drastically simplify computations: the upshot of the condition is that the family is completely determined by its behaviour over a $d$-dimensional disc $\Spec \mathbb{C} \series{d}$, and consequently that all computations can be reduced to the case where $X$ is the most convenient $d$-dimensional variety, in practice often $\BBA^d$. Moreover, just as a factorization space over any variety $X$ can be used to produce examples of factorization algebras over $X$ by pushing forward line bundles on the factorization space that are compatible with the factorization structure, so can a universal factorization space be used to produce a universal factorization algebra: one only needs to check that the line bundles in question are compatible with the compatibility isomorphisms comprising the universality of the factorization space. Alternatively, given a family of factorization algebras, one may check that it is universal by constructing it from such a compatible family of line bundles on a universal factorization space, these latter conditions being perhaps more straightforward to verify. 

This strategy is employed for example by Kapranov and Vasserot in \cite{KV}, where they claim that the computation of the chiral de Rham complex can be reduced to the case of the curve $\BBA^1$. It turns out that the exposition in \cite{KV} implicitly uses an incorrect definition of a universal factorization space, and relies on the existence of maps which are in general undefined except on an open neighbourhood of the diagonal; however as we will see in section \ref{sec: examples}, if the correct definition as presented in this paper is used, the domain of definition of these maps is large enough to imply that the factorization space in question is universal, and consequently that the computations over $\BBA^1$ do indeed suffice. Indeed, let $X$ be an affine scheme, let $C$ be a curve, and recall the factorization space $\mathcal{L}(X)_{\Ran{C}}$ of meromorphic jets from $C$ to $X$ introduced in \cite{KV}; then we prove the following:

\begin{prop*}[Example \ref{eg: meromorphic jets KV}]
The assignment $C \mapsto \mathcal{L}(X)_{\Ran{C}}$ defines a universal factorization space in dimension one. 
\end{prop*}

As suggested by the approach of \cite{KV}, the expectation that a category of universal factorization spaces or algebras should exist is not a new one, but the definitions are slightly more subtle than one would na\"ively expect. The obvious way to begin is to say that a universal factorization algebra must be an assignment
\begin{align*}
X, \text{ a $d$-dimensional variety} \mapsto \CA_{\Ran{X}}, \text{ a factorization algebra over $X$}, 
\end{align*}
together with data giving compatibilities between pullbacks of factorization algebras along \'etale maps $X \to Y$. However, for this to make sense, we need to define what we mean by the \emph{pullback} of a factorization algebra $\CA_{\Ran{Y}}$ over $Y$ by an \'etale map $\phi: X \to Y$. It turns out that this is not written explicitly anywhere in the literature: to compute it, we must consider the chiral algebra $\CB_Y$ associated to $\CA_{\Ran{Y}}$, take its pullback $\phi^*\CB_Y \defeq \CB_X$, and define $\phi^*(\CA_{\Ran{Y}})$ to be the factorization algebra $\CA_{\Ran{X}}$ associated to the chiral algebra $\CB_X$. 

In this paper we give a description of $\phi^*(\CA_{\Ran{Y}})$ without making use of the equivalence between factorization algebras and chiral algebras, thus giving a more hands-on construction of the factorization algebra. This allows us to formulate the definition of a universal factorization algebra; we also have a non-linear analogue, in the language of factorization spaces. The key idea is the intuitive observation that the interesting information of a factorization space $\CY_{\Ran{Y}} = \left\{\CY_{X^I} \to X^I\right\}$ is contained entirely in the data of
\begin{enumerate}
\item $\CY_X \to X$, and
\item the information of how to glue copies of $\CY_X$ together as we approach the diagonal $\Delta^I(X)$ in $X^I$---that is, the restriction of the factorization isomorphisms to open neighbourhoods of the diagonal.
\end{enumerate}

We formalize this intuition by introducing the notion of a \emph{weak factorization space}, where we only require the data of the spaces $\CY_{X^I}$ and the structure isomorphisms of a factorization space to be given close to the diagonal, and by proving the following:

\begin{thm*}[Theorem \ref{thm: weak factorization spaces are factorization spaces}, Theorem \ref{thm: weak factorization algebras}]
The forgetful functor from the category of factorization spaces over $X$ to the category of weak factorization spaces over $X$ is an equivalence of categories.

Similarly, the functor from the category of factorization algebras over $X$ to the category of weak factorization algebras over $X$ is an equivalence of categories.
\end{thm*}

It turns out to be much easier to define the pullback of a weak factorization space or algebra than that of an ordinary one, but this equivalence allows us to extend the definition. With these definitions in hand, we can introduce universal factorization spaces and factorization algebras; we can furthermore check that our definition is compatible with the more familiar notion of a universal chiral algebra. 

The structure of the paper is as follows. We begin in section \ref{sec: background} by recalling the definitions of factorization spaces and algebras, and chiral algebras, and fixing notation that will be used in the rest of the paper. In section \ref{sec: preliminary definition} we propose a na\"ive definition of the pullback of a factorization space, which is the definition that was used implicitly in e.g. \cite{KV}. We explain why it is not a good definition, to understand why the definition we will eventually work with needs to be more subtle.

In section \ref{sec: weak factorization spaces}, we will introduce the notion of a \emph{weak} factorization space or algebra, and show that the categories of weak and ordinary (non-weak) factorization spaces over a fixed variety $X$ are equivalent. In section \ref{sec: weak factorization algebras} we extend these definitions and arguments to the case of weak factorization algebras.  We will see in section \ref{sec: pullback of factorization spaces} that it is straightforward to define the pullback of a weak factorization space or algebra along an \'etale morphism. This allows us to define the pullback of a factorization space (or algebra) by viewing it as a weak factorization space (or algebra, respectively), and applying the pullback functor in that category.

We will conclude with some remarks justifying these definitions. In section \ref{sec: examples} we formulate carefully the notion of a universal factorization space. It is then straightforward to verify that some common and important examples of factorization spaces, namely, the Beilinson-Drinfeld Grassmannian (introduced in \cite{BD2}), and the jet-spaces studied by Kapranov and Vasserot in \cite{KV}, form universal factorization spaces, as expected. We also provide an example of a universal factorization space in arbitrary dimension $d$, constructed using the Hilbert scheme of points. Finally, in section \ref{sec: pullback of factorization and chiral algebras}, we will show that our definition of the pullback of a factorization algebra agrees with the definition obtained by pulling back the corresponding chiral algebra.

\subsection{Conventions and notation}
\label{subsec: conventions}
We fix $k = \overline{k}$, an algebraically closed field of characteristic zero, and work with the category $\Sch$ of schemes of finite type over $k$. We fix a dimension $d$ and use capital Roman letters $X, Y \ldots$ to denote smooth varieties over $k$ of dimension $d$. For a given such variety $X$, we will denote by $\CD(X)$ the \emph{infinity} or \emph{d.g. category} of $\CD$-modules on $X$. Functors between such categories should always be understood to be functors of $(\infty,1)$-categories. 

Given a variety $X$, we will work with its \emph{Ran space}, $\Ran{X}$. This is defined as follows. We let $\fSet$ denote the category of finite non-empty sets and surjections. Given such a surjection $\alpha: I \surj J$, we obtain a diagonal embedding of the Cartesian products of $X$: $\Delta_X(\alpha) = \Delta(\alpha): X^J \emb X^I$ (we suppress the space $X$ from the notation when no confusion will result). Using the Yondeda embedding, we view each $X^I$ as a contravariant functor from $\Sch^\op$ to the category of $\infty$-groupoids, i.e. as a \emph{prestack}. We can then take the colimit inside the category of prestacks, and define this to be the Ran space of $X$:
\begin{align*}
\Ran{X} \defeq \colim_{I \in \fSet^\op} X^I,
\end{align*}
with the structure maps $X^I \to \Ran{X}$ denoted by $\Delta_{X^I}$. The Ran space is not representable as a scheme (it is a \emph{pseudo-indscheme}), but it is still possible to study its geometry. In particular, its category of $\CD$-modules can be identified with the limit $\lim_{I \in \fSet} \CD(X^I)$ (over the maps $\Delta(\alpha)^!: \CD(X^I) \to \CD(X^J)$), and also with the colimit $\colim_{I \in \fSet^\op} \CD(X^I)$ (over the maps $\Delta(\alpha)_{*} : \CD(X^J) \to \CD(X^I)$). We let $\Delta_{X^I}^!$ and $(\Delta_{X^I})_*$ denote the structural functors of the limit and the colimit descriptions, respectively. For more details on the d.g. category of $\CD$-modules on the Ran space, see for example section 2.1 of \cite{FG}.

\subsection{Acknowledgments}
The idea to formalize the notion of a universal factorization space was inspired in part by Kapranov and Vasserot's paper \cite{KV}. In particular, I thank Mikhail Kapranov for a helpful conversation in which he made suggestions leading to the definition of a weak factorization space. I also thank Tom Nevins for his valuable advice, and Kobi Kremnitzer for numerous enjoyable discussions on this topic. I am grateful to the reviewer for a number of constructive remarks and recommendations.

\section{Background: factorization spaces}
\label{sec: background}
Let us begin by fixing some basic notation and recalling some essential definitions that will be used throughout the paper. Our primary references for chiral algebras and factorization algebras are \cite{BD1} and \cite{FG}, wherein many more details can be found. 

\begin{defn}\label{factorization space}
A \textit{factorization space} over $X$ consists of the following data:
\begin{enumerate}
\item A prestack $\CY_{\Ran{X}}$ expressible as a colimit over $\fSet^\op$:
\begin{align*}
\CY \simeq \colim_{I\in\fSet^\op}\CY_{X^I},
\end{align*}
where for each $I \in \fSet$, $\CY_{X^I} \xrightarrow{f_I} X^I$ is an indscheme over $X^I$ equipped with a formally integrable connection, and for any $\alpha: I \surj J$, there is an indproper morphism $\CY(\alpha): \CY_{X^J} \to \CY_{X^I}$ compatible with the maps $f_I$, $f_J$ and $\Delta(\alpha)$. 
\item \emph{Ran's condition:} For any surjection $\alpha: I \surj J$, there is a natural map $\nu_{\alpha}: \CY_{X^J} \to X^J \times_{X^I} \CY_{X^I}$ given by
\begin{center}
\begin{tikzpicture}[>=angle 90]
\matrix(b)[matrix of math nodes, row sep=2em, column sep=2em, text height=1.5ex, text depth=0.25ex]
{\CY_{X^J} &                          &         \\
         & X^J \times_{X^I} \CY_{X^I} & \CY_{X^I} \\
         & X^J                      & X^I     \\};
\path[->, font=\scriptsize]
 (b-1-1) edge[bend left=20] node[above right]{$\CY(\alpha)$} (b-2-3)
 (b-1-1) edge[bend right=40] node[below left]{$f_J$} (b-3-2)
 (b-2-2) edge (b-2-3)
 (b-2-2) edge (b-3-2)
 (b-2-3) edge node[right]{$f_I$} (b-3-3)
 (b-3-2) edge node[below]{$\Delta(\alpha)$} (b-3-3);
\path[->, font=\scriptsize]
 (b-1-1) edge[dashed] node[above right]{$\nu_\alpha$} (b-2-2);
\end{tikzpicture}
\end{center}
We require that $\nu_\alpha$ be an equivalence of indschemes, and that $\nu$ be associative in the obvious sense.\footnote{In particular, it follows that the morphisms $\CY(\alpha)$ are contravariantly functorial in $\alpha$.} 
\item \emph{Factorization:} Given $\alpha: I \surj J$ as above, we obtain a partition of $I$ as $\bigsqcup_{j\in J} I_j$, where $I_j = \left\{ i \in I \ \vert \ \alpha(i) = j \right\}$, and consider the following open subscheme of $X^I$:
\begin{align*}
U = U(\alpha) \defeq \left\{ (x_i)_{i \in I} \in X^I \ \vert\ x_{i_1} \ne x_{i_2} \text{ unless } \alpha(i_1) = \alpha(i_2) \right\}.
\end{align*}
We let $j=j(\alpha)$ denote the open embedding $U \emb X^I \cong \prod_{j\in J} X^{I_j}$, and consider the following two pullback diagrams:
\begin{center}
\begin{tikzpicture}[>=angle 90]
\matrix(c)[matrix of math nodes, row sep=3em, column sep=1.6em, text height=1.5ex, text depth=0.25ex]
{U \times_{X^I}\CY_{X^I} & \CY_{X^I} && U \times_{X^I} \left(\prod_{j\in J} \CY_{X^{I_j}}\right) & \prod_{j \in J} \CY_{X^{I_j}} \\
 U         & X^I  && U                          & \prod_{j \in J} X^{I_j}\\};
\path[->, font=\scriptsize]
 (c-1-1) edge node[above]{$j^\prime$}(c-1-2)
 (c-1-1) edge node[left]{$f_I^\prime$}(c-2-1)
 (c-1-2) edge node[right]{$f_I$} (c-2-2)
 (c-1-4) edge node[above]{$j^{\prime\prime}$}(c-1-5)
 (c-1-4) edge node[left]{$\left(\prod_{j \in J}f_{I_j}\right)^{\prime\prime}$}(c-2-4)
 (c-1-5) edge node[right]{$\prod_{j \in J}f_{I_j}$} (c-2-5);
\path[right hook->, font=\scriptsize]
 (c-2-1) edge node[below]{$j$} (c-2-2)
 (c-2-4) edge node[below]{$j$} (c-2-5);
\end{tikzpicture}
\end{center}
We require an equivalence 
\begin{align*}
d_\alpha :  U \times_{X^I} \left(\prod_{j\in J} \CY_{X^{I_j}}\right) \EquivTo U \times_{X^I} \CY_{X^I}
\end{align*}
of indschemes over $U$. Moreover, these equivalences $d_\alpha$ should be associative and compatible with the other structure maps $\nu_\alpha$. 
\end{enumerate}
\end{defn}

\begin{eg}
\label{eg: compatibility of d}
Let us write out the compatibility condition between different $d_\alpha$ explicitly. Suppose that we have surjections of finite sets as follows:
\begin{align*}
I \xrightarrowdbl{\beta} K \xrightarrowdbl{\gamma} J.
\end{align*}
Let $\alpha$ denote the composition $\gamma \circ \beta$, and furthermore fix the notation:
\begin{align*}
K = \bigsqcup_{j \in J} K_j; & \qquad I = \bigsqcup_{j \in J} I_j;
\end{align*}
\begin{align*}
\beta_j: I_j \surj K_j.
\end{align*}
Notice that $U(\beta) \subset \prod_{j \in J} U(\beta_j)$, and also $U(\beta) \subset U(\alpha)$. This allows us to restrict $\prod_{j \in J} d_{\beta_j}$ and $d_\alpha$ to $U(\beta)$, so that the following composition is well-defined:
\begin{align*}
\left(\prod_{k \in K} \CY_{X^{I_k}} \right) \vert_{U(\beta)} \xrightarrow{\prod_{j \in J} d_{\beta_j}} \left(\prod_{j \in J} \CY_{X^{I_j}} \right) \vert_{U(\beta)} \xrightarrow{d_{\alpha}} \left(\CY_{X^I} \right) \vert_{U(\beta)}. 
\end{align*}
The compatibility condition is simply that this composition is equal to $d_{\beta}$. 
\end{eg}

For comparison, let us now give a rough definition of a factorization algebra. A more rigorous definition will be given in section \ref{sec: weak factorization algebras}.
\begin{rough defn}\label{rough defn: factorization algebra}
A \emph{factorization algebra} $\CA$ on $X$ is the linear analogue of a factorization space: it consists of a family $\left\{\CA_{X^I} \in \CD(X^I) \right\}$ of $\CD$-modules\footnote{Let us emphasize that by $\CD(Y)$ we will always mean the d.g. category of $\CD$-modules on $Y$, and in particular that all functors of $\CD$-modules are the derived versions.} together with isomorphisms 
\begin{align*}
\nu_\alpha&: \CA_{X^I} \to \Delta(\alpha)^! \CA_{X^J}; \\ d_\alpha&:  j(\alpha)^* \left(\boxtimes_{j \in J} \CA_{X^{I_j}} \right) \to j(\alpha)^* \CA_{X^I}
\end{align*}
for any $\alpha: I \surj J$. Since we are working in the d.g. categories, the compatibilities between these different isomorphisms $\nu_\bullet$, $d_\bullet$ are additional data consisting of equivalences between various compositions, rather than simply equalities as in the case of factorization spaces. These equivalences are themselves subject to higher coherence requirements. We will make this more precise in section \ref{sec: weak factorization algebras}. 
\end{rough defn}

\begin{defn}\label{defn: chiral algebra}
A \emph{chiral algebra} on $X$ is a $\CD$-module $\CB_X$ on $X$ together with the structure of a Lie algebra object on the $\CD$-module 
\begin{align*}
\CB_{\Ran X} \defeq (\Delta_{X})_ *(\CB_X) \in (\CD(\Ran X), \otimes^\ch).
\end{align*}
(For details on the definition of the chiral tensor product $\otimes^\ch$ see section 2.3 of \cite{FG}; see also section \ref{sec: weak factorization algebras} of the current paper where we define an analogous tensor product on the category of $\CD$-modules over an open subspace $W$ of the Ran space.)
\end{defn}

Somewhat more specifically, we require a morphism of sheaves on $\Ran{X}$
\begin{align*}
\mu_{\CB}: \CB_{\Ran{X}} \otimes^\ch \CB_{\Ran{X}} \to \CB_{\Ran{X}}
\end{align*}
together with higher isomorphisms corresponding to skew-symmetry and the Jacobi identity. 

In particular, considering the restriction of this map along the canonical map $X^2 \to \Ran{X}$ and using the definition of $\otimes^\ch$, we have a morphism of sheaves on $X^2$
\begin{align*}
j_{*} j^* \left(\CB_X \boxtimes \CB_X\right) \to \Delta_{!}\CB_X,
\end{align*}
which we will also denote by $\mu_\CB$. (Here $\Delta: X \to X^2$ is the diagonal embedding and $j = j(\id)$ is the complementary open embedding.)

\section{A preliminary definition for \'etale pullback of a factorization space}
\label{sec: preliminary definition}
Let $\left\{ \CY_{Y^I} \to Y^I \right\}$ be a factorization space over a smooth variety $Y$, and let $\phi: X \to Y$ be an \'etale morphism. We wish to define a factorization space $\left\{ \CY_{X^I} \to X^I \right\}$ over $X$, the \emph{pullback of $\CY_{\Ran{Y}}$ along $\phi$}. The first thing we could try is the following:
\begin{align*}
\CY_{X^I} \defeq X^I \times_{Y^I} \CY_{Y^I},
\end{align*}
where the map $X^I \to Y^I$ is just the $I$-fold product of $\phi$. However, in general, this does not give a factorization space.

Indeed, consider the set 
\begin{align*}
Z_\phi \defeq \left\{(x_1,x_2)\in X^2 \ \vert \ x_1 \ne x_2, \text{ but } \phi(x_1)=\phi(x_2) \right\}.
\end{align*}

For a general \'etale morphism $\phi$, this set is non-empty. Suppose that $(x_1,x_2)$ is a point of $Z_\phi$, and let $y = \phi(x_1) = \phi(x_2) \in Y$. Consider the fibre of $\CY_{X^2}$ over $(x_1, x_2)$: by definition, it is the fibre $\CY_{Y^2,(y,y)}$ of $\CY_{Y^2}$ over the point $\phi^2(x_1,x_2)=(y,y)$. By assumption, $\CY_{Y^2} \vert_{\Delta(Y)}$ is isomorphic to $\CY_Y$, so that this fibre is $\CY_{Y,y}$. 

However, if $\left\{\CY_{X^I}\right\}$ were in fact a factorization space, we would have 
\begin{align*}
\CY_{X^2,(x_1, x_2)} \simeq \CY_{X, x_1} \times \CY_{X, x_2} \simeq \CY_{Y, y} \times \CY_{Y, y}.
\end{align*}
It follows that $\left\{\CY_{X^I} \right\}$ defined as above does \emph{not} give a factorization space unless $\phi$ is injective. 

\begin{observation}
Note, however, that the axioms of a factorization space only fail to hold on the set $Z_\phi$: it is straightforward to check that Ran's condition holds on $\Delta(X) \subset X \times X$ and that the factorization condition is satisfied on $U_X \setminus Z_\phi$. 

Note also that because $\phi$ is \'etale, $X \to X \times_Y X$ is an open embedding. It follows that $Z_\phi = X \times_Y X \setminus X$ is closed in $X \times_Y X$, and hence also in $X \times X$. Therefore, the complement $V_\phi$ of $Z_\phi$ gives an open neighbourhood of the diagonal in $X\times X$.  

Recall from the introduction that the interesting data of a factorization seems to be concentrated near the diagonal. Thus, although our definition of the pullback did not work over all of $X^I$, there is reason to hope that it is a good definition on an open subscheme of $X^I$ for each $I$, and that this data is enough to completely determine the rest of the definition. 

We formalize this intuition in the following section. 
\end{observation}

\section{Weak factorization spaces}
\label{sec: weak factorization spaces}
\begin{defn}
\label{def: weak factorization space}
A \emph{weak factorization space over $X$} is given by the following data:

\begin{enumerate}
\item For each finite set $I$, we require an open subscheme $W(I) \subset X^I$, containing the diagonal $\Delta(X) \subset X^I$. We require an indscheme $g_I: \CZ_I \to W(I)$, equipped with a formally integrable connection over $W(I)$.
\item For any surjection $\alpha: I \surj J$, we require an open subscheme $R(\alpha)$ of $W(J) \bigcap \Delta(\alpha)^{-1}(W(I))$ in $X^J$, containing the diagonal $\Delta(X)$. We require an isomorphism $\tilde{\nu}_\alpha$ between the restrictions of $\CZ_I$ and $\CZ_J$ to $R(\alpha)$. In other words, Ran's condition must hold on $R(\alpha)$. 
\item For any surjection $\alpha: I \surj J$, giving rise to a partition of $I$ as $\bigsqcup_{j \in J} I_j$, we require an open subscheme $F(\alpha)$ of $W(I) \bigcap \left(\prod_{j \in J} W(I_j)\right)$ in $X^I$, containing the diagonal $\Delta(X)$. We require an isomorphism $\tilde{d}_\alpha$ of the restrictions of $\CZ_{I}$ and $\prod_{j \in J} \CZ_{I_j}$ to $F(\alpha) \bigcap U(\alpha)$. In other words, the factorization condition must hold on $F(\alpha)$. 
\item We require compatibilities between the morphisms $\tilde{\nu}$ and $\tilde{d}$ with each other and under composition of surjections, wherever these compositions make sense. \footnote{In fact, one could weaken this condition, and require only that the compatibilities exist only on an open subset of the range where the compositions are defined, again containing the diagonal. The resulting category turns out to be equivalent to the one defined here, but requires extra layers of notation, so we will focus on the notationally simpler version, which is sufficiently general for our purposes.}
\end{enumerate}
\end{defn}

\begin{defn}
\label{def: morphism of weak factorization spaces}
A \emph{morphism} $F$ between two weak factorization spaces 
\begin{align*}
(\CZ_I, W(I), \ldots) \text{ and } (\CZ^\prime_I, W^\prime(I), \ldots)
\end{align*}
is a collection of morphisms $F_I: \CZ_I\vert_{V(I)} \to \CZ^\prime\vert_{V(I)}$ over some open subschemes $V(I) \subset W(I) \bigcap W^\prime(I)$ which are required to contain the small diagonal. These morphisms must be compatible with the morphisms $\tilde{\nu}, \tilde{\nu}^\prime$ and $\tilde{d}, \tilde{d}^\prime$ wherever the compositions make sense.

Let $\WFSp{X}$ denote the category of weak factorization spaces over $X$. 
\end{defn}

In particular, a factorization space $\left\{\CY_{X^I} \to X^I \right\}$ is a weak factorization space, where we can take $W(I)=X^I$ for each $I$. We can also take $R(\alpha)$ to be all of $X^J$, and $F(\alpha)$ to be $X^I$ for each surjection $\alpha: I \surj J$. Furthermore, a morphism of factorization spaces yields a morphism of the corresponding weak factorization spaces, where we can take $V(I)$ to be all of $X^I$ for each $I$. 

\begin{defn}
In other words, we have a forgetful functor 
\begin{align*}
\Weak: \FSp{X} \to \WFSp{X}.
\end{align*}
\end{defn}

\begin{thm}
\label{thm: weak factorization spaces are factorization spaces}
The functor $\Weak$ is an equivalence of the categories of ordinary and weak factorization spaces.
\end{thm}

\begin{proof}
We will show that $\Weak$ is an equivalence by exhibiting a quasi-inverse,
\begin{align*}
\Glue: \WFSp{X} \to \FSp{X}.
\end{align*}

Let $\CZ = (\CZ_I, W(I), \ldots)$ be a weak factorization space. Our goal is to build a factorization space $\Glue(\CZ) = \left\{\CY_{X^I} \to X^I\right\}$ by gluing together the pieces of $\CZ$ along the isomorphisms $\tilde{\nu}$ and $\tilde{d}$. We will do this by induction on $\vert I \vert$. 

The case $I= \left\{\pt\right\}$ is trivial: we have $W(\pt) = X$, and we take $\CY_{X} \defeq \CZ_{\left\{\pt\right\}}$.

Let us also carry out the case $I = \left\{1, 2\right\}$ explicitly, to motivate the induction step. First notice that we have an open cover of $X^2$ given by $F(\alpha) \cup U$, where $\alpha = \id_I$ and $U=U(\alpha)= X^2 \setminus \Delta(X)$. So to define a space $\CY_I$ on $X^2$ it suffices to define a space on each of $F= F(\alpha)$ and $U$, and then to provide an isomorphism of these spaces over the intersection. (That indschemes can be defined locally in this way, as we are used to doing for schemes, follows in a straightforward way from the characterization of an indscheme as a stack which is the filtered colimit of its closed subschemes.) It is clear how to proceed: we take $(\CZ_I)\vert_{F}$ over $F$, and $\left(\CY_X \times \CY_X\right)\vert_{U}$ over $U$. Then the isomorphism is given by $\tilde{d}_{\alpha}$, using the fact that $\CY_X = \CZ_{\left\{\pt\right\}}$. 

Let us now take $n \ge 3$ and assume that we have constructed the spaces $\CY_{X^K} \to X^K$ for all $K$ with $\vert K \vert \le n-1 $, and moreover that we have constructed the isomorphisms $\nu_\alpha$ and $d_\alpha$ for all surjections between sets of size at most $n-1$. We suppose that the $\nu_\alpha$ and $d_\alpha$ satisfy the compatibility conditions with all $\nu$ and $d$ already defined.  Finally, we assume that if $F(K)$ denotes the intersection of all $F(\beta)$ over all surjections $\beta$ from $K$ to a strictly smaller set, we have $\CY_{X^K} \vert_{F(K)} = \CZ_{K} \vert_{F(K)}$. Let $I$ be a finite set of size $n$. 

We would like to use the same idea as in the case $n=2$. To begin, we need to find an open cover of $X^I$. Let $F(I)$ denote the intersection of the open sets $F(\alpha)$ for each surjection $\alpha$ from $I$ to a strictly smaller set; $F(I)$ is a finite intersection of open neighbourhoods of the diagonal $\Delta(X)$ in $X^I$, so it is again an open set containing the diagonal. We cannot express the complement $X^I \setminus \Delta(X)$ as a set $U(\alpha)$ for any particular $\alpha$, but we notice instead that it is the union of all the sets $U(\alpha)$ where $\alpha$ runs over all surjections from $I$ to any set of size at least $2$.

Next we need to specify the components of the space $\CY_{X^I}$ living over each piece of the open cover. Over $F(I)$, we take the restriction of $\CZ_I$ to $F(I)$. Over $U(\alpha)$, we take the restriction of $\prod_{j \in J} \CY_{X^{I_j}}$ to $U(\alpha)$. Note that because of the assumption that $\vert J \vert \ge 2$, each $I_j$ has size strictly less than $n$, and hence $\CY_{X^{I_j}}$ is defined, by the induction hypothesis. 

The next step is to provide isomorphisms between these pieces on the intersections of any pair of sets in the open cover. First suppose we have $\alpha: I \surj J$, $\beta: I \surj K$, and consider $U(\alpha) \cap U(\beta)$. Define
\begin{align*}
J \star K \defeq \{ (j,k) \in J \times K \ \vert\ I_j \cap I_k \ne \emptyset \}.
\end{align*}
By construction, there is a surjection from $I$ to $J \star K$, which we will denote by $\alpha \star \beta$; moreover, the maps $\alpha$ and $\beta$ obviously factor through $\alpha \star \beta$. Let us denote by $I_{jk}$ the intersection $I_j \cap I_k$, whenever it is non-empty; it is of course equal to $I_{(j,k)}$. Notice that $U(\alpha) \cap U(\beta) = U(\alpha \star \beta)$. Let us also fix the following notation: for $k \in K$,
\begin{align*}
J(k) &\defeq \left\{j \in J \ \vert\ (j,k) \in J \star K \right\};\\
\alpha_k & \defeq \alpha\vert_{I_k} : I_k \surj J(k).
\end{align*}
Similarly, for fixed $j \in J$, we define a subset $K(j)$ of $K$, and the restriction $\beta_j$ of $\beta$ to $I_j$.

To define the isomorphism $\phi_{\alpha,\beta}$ between the restrictions of the spaces 
\begin{align*}
\left(\prod_{j \in J} \CY_{X^{I_j}} \right) \vert_{U(\alpha \star \beta)} \quad \text{ and}   \quad \left(\prod_{k \in K} \CY_{X^{I_k}} \right) \vert_{U(\alpha \star \beta)},
\end{align*}
we will define an isomorphism between each of these and $\left(\prod_{(j,k) \in J \star K} \CY_{X^{I_{jk}}} \right) \vert_{U(\alpha \star \beta)}$. Indeed, notice that $U(\alpha \star \beta) \subset \prod_{k \in K} U(\alpha_k)$; it follows that
\begin{align*}
\left(\prod_{k \in K} \CY_{X^{I_k}} \right) \vert_{U(\alpha \star \beta)} & = \left(\prod_{k \in K} \left(\CY_{X^{I_k}} \right) \vert_{U(\alpha_k)} \right) \vert_{U(\alpha \star \beta)}.
\end{align*}

Now by the induction hypothesis we have isomorphisms 
\begin{align*}
d_{\alpha_k} : \left(\prod_{j \in J(k)} \CY_{X^{I_{jk}}} \right) \vert_{U (\alpha_k)} \EquivTo (\CY_X^{I_k}) \vert_{U(\alpha_k)}.
\end{align*}
Taking the product of the $d_{\alpha_k}$ over all $k \in K$ and restricting to $U(\alpha \star \beta)$ gives an isomorphism from
\begin{align*}
\left(\prod_{k \in K} \left(\prod_{j \in J(k)} \CY_{X^{I_{jk}}} \right) \vert_{U(\alpha_k)} \right) \vert_{U (\alpha \star \beta)} = \left( \prod _{(j,k) \in J \star K} \CY_{X^{I_{jk}}} \right)_{U(\alpha\star\beta)}
\end{align*}
to 
\begin{align*}
\left(\prod_{k \in K} \CY_{X^{I_k}} \right) \vert_{U(\alpha \star \beta)}.
\end{align*}

Let us denote this isomorphism by $\phi_{\alpha}^\beta$. Swapping the roles of $J$ and $K$, we also obtain an isomorphism $\phi_{\beta}^\alpha$, and we define the desired compatibility isomorphism $\phi_{\alpha,\beta}$ to be the composition $\phi^\beta_\alpha \circ (\phi_\beta^\alpha)^{-1}$.

It is immediate from this construction that $\phi_{\alpha, \alpha} = \id$, and that $\phi_{\beta, \alpha} = \phi_{\alpha, \beta}^{-1}$. The remaining compatibility condition to check is the compatibility of the isomorphisms on triple overlaps: we need to show that
\begin{align*}
\phi_{\beta, \gamma} \circ \phi_{\alpha, \beta} = \phi_{\alpha, \gamma}
\end{align*}
on $U(\alpha \star \beta \star \gamma)$. 

For this we use the compatibility of the morphisms $d_\alpha$ with respect to composition of the surjections $\alpha$. More specifically, we have the following commutative diagram (where all spaces and morphisms are restricted to $U(\alpha \star \beta \star \gamma)$, and all morphisms are isomorphisms, although we have omitted this from the notation):
\begin{center}
\begin{tikzpicture}[>=angle 90]
\matrix(a)[matrix of math nodes, row sep=3em, column sep=2em, text height=1.5ex, text depth=0.25ex]
{                                            & \prod_{k \in K} \CY_{X^{I_k}}                           &                                \\
\prod_{(j,k) \in J \star K} \CY_{X^{I_{jk}}} &                                                         & \prod_{(k,l) \in K \star L} \CY_{X^{I_{kl}}} \\ 
                                             & \prod_{(j,k,l) \in J \star K \star L} \CY_{X^{I_{jkl}}} &                                \\
\prod_{j \in J} \CY_{X^{I_j}}                &                                                         & \prod_{l \in L} \CY_{X^{I_l}}  \\
                                             & \prod_{(j,l) \in J \star L} \CY_{X^{I_{jl}}}            &                                \\};
\path[->, font=\scriptsize]
 (a-3-2) edge node[above right]{$\prod_{j,k} d_{\gamma \vert_{I_{jk}}}$} (a-2-1)
         edge node[right]{$\prod_{k} d_{\alpha \star \gamma \vert_{I_k}}$} (a-1-2)
         edge node[below right]{$\prod_{k,l} d_{\alpha \vert_{I_{kl}}}$} (a-2-3)
         edge node[below]{$\prod_{l} d_{\alpha \star \beta \vert_{I_l}}$} (a-4-3)
         edge node[left]{$\prod_{j,l} d_{\beta \vert_{I_{jl}}} $} (a-5-2)
         edge node[above left]{$\prod_{j} d_{\beta \star \gamma \vert_{I_j}}$} (a-4-1)
 (a-2-1) edge node[above left]{$\prod_k d_{\alpha \vert_{I_k}}$} (a-1-2)
         edge node[left]{$\prod_{j} d_{\beta \vert_{I_j}}$} (a-4-1)
 (a-2-3) edge node[above right]{$\prod_k d_{\gamma \vert_{I_k}}$} (a-1-2)
         edge node[right]{$\prod_{l} d_{\beta \vert_{I_l}}$} (a-4-3)
 (a-5-2) edge node[below left]{$\prod_{j} d_{\gamma \vert_{I_j}}$} (a-4-1)
         edge node[below right]{$\prod_{l} d_{\alpha \vert_{I_l}}$} (a-4-3);
\end{tikzpicture}
\end{center}

The commutativity of each of the six triangles follows precisely from the compatibility condition described in Example \ref{eg: compatibility of d}. For example, the composition of surjections
\begin{align*}
I_k \xrightarrowdbl{\alpha \star \gamma \vert_{I_k}} (J \star L)(k) \surj J(k)
\end{align*}
is equal to $\alpha\vert_{I_k}$, and Example \ref{eg: compatibility of d} implies that
\begin{align*}
d_{\alpha \star \gamma \vert_{I_k}} = d_{\alpha \vert_{I_k}} \circ \prod_{j \in J(k)} d_{\gamma \vert_{I_{jk}}}.
\end{align*}
It follows that the two ways of tracing around the outside of the diagram from $\prod_{j \in J} \CY_{X^{I_j}}$ to $\prod_{l \in L} \CY_{X^{I_l}}$ are equal. But going around the top is, by definition, $\phi_{\beta, \gamma} \circ \phi_{\alpha, \beta}$, while going along the bottom gives $\phi_{\alpha, \gamma}$. 

Finally, we need to define compatibility isomorphisms on the overlaps
\begin{align*}
U(\alpha)_0 \defeq U(\alpha) \cap F(I).
\end{align*}
Since $U(\alpha)_0$ is contained in $F(\alpha) \cap U(\alpha)$, we have the weak factorization isomorphism
\begin{align*}
\tilde{d}_\alpha: \left(\prod_{j \in J} \CZ_{I_j}\right) \vert_{U(\alpha)_0} \EquivTo \left(\CZ_I\right) \vert_{U(\alpha)_0}. 
\end{align*}
On the other hand, by the induction hypothesis, we have
\begin{align*}
\left(\prod_{j \in J} \CZ_{I_j}\right) \vert_{U(\alpha)_0} = \left(\prod_{j \in J} \CY_{X^{I_j}} \right) \vert_{U(\alpha)_0}.
\end{align*}
So we can take $\phi_{\alpha,0}$ to be $\tilde{d}_\alpha$, and $\phi_{0,\alpha}$ to be its inverse. Compatibility of the morphisms $\tilde{d}$ with respect to composition ensures in a similar way to the above arguments that these isomorphisms are compatible on triple overlaps $U(\alpha \star \beta)_0$. 

Therefore, we have succeeded in building a space $\CY_{X^I}$ over $X^I$ which satisfies 
\begin{align*}
\left(\CY_{X^I}\right) \vert_{F(I)} = \left(\CZ_I\right) \vert_{F(I)},
\end{align*}
and which comes equipped with the necessary isomorphisms $d_\alpha$ and $\nu_\alpha$. 

This completes the induction step, and hence the construction of $\Glue(\CZ)$. It is clear from the construction that a morphism
\begin{align*}
\CZ \to \CZ^\prime
\end{align*}
of weak factorization spaces gives rise to a morphism $\Glue(\CZ) \to \Glue(\CZ^\prime)$ of factorization spaces. It is also immediate that $\Weak \circ \Glue$ is equivalent to the identity functor on $\WFSp{X}$, and conversely that $\Glue \circ \Weak$ is equal to the identity functor on $\FSp{X}$. 
\end{proof}

\section{Weak factorization algebras}
\label{sec: weak factorization algebras}
In this section, we will see that we can make analogous definitions for factorization algebras. We again have a forgetful functor from the category of factorization algebras over $X$ to the category of weak factorization algebras over $X$. The proof that it is an equivalence is almost completely parallel to the above; the key difference is that the factorization and Ran isomorphisms for a factorization algebra are only required to be compatible \emph{up to natural isomorphisms}, which are themselves required to satisfy higher compatibilities, whereas for a factorization space, the compatibilities are strict. This means that in gluing the pieces of a weak factorization algebra to get an ordinary factorization algebra, we must check compatibility conditions over overlaps of multiple sets of the open cover, not just double and triple overlaps. However, since our open covers are all finite, this process does terminate and we can conclude that all the desired compatibility isomorphisms exist and satisfy the required properties. Let us spell out these subtleties more carefully. We begin by reviewing the definition of a factorization algebra over a variety $X$ given by Francis and Gaitsgory in \cite{FG}, so that we can see how to adjust it to give a weak version.

\begin{defn}[\cite{FG} 2.4.7]\label{defn: factorization algebra}
First, we require the data of a $D$-module $\CA$ over $\Ran{X}$, which is a cocommutative coalgebra object in the symmetric monoidal category $(\CD(\Ran{X}), \otimes^\ch)$. The comultiplication map is denoted 
\begin{align*}
c: \CA \to \CA \otimes^\ch \CA;
\end{align*}
pulling back along any $\Delta_I: X^I \to \Ran{X}$ gives a map
\begin{align*}
\CA_I \to \Delta_I^!(\CA \otimes^\ch \CA) \simeq \bigoplus_{I = I_1 \sqcup I_2} j(\alpha)_* j(\alpha)^* \left( \CA_{I_1} \boxtimes \CA_{I_2}\right),
\end{align*}
where the sum is over all partitions of $I$ into two pieces, and $\alpha$ is the corresponding surjection $I \surj \{ 1, 2 \}$.

Adjunction induces for each such $\alpha$ a map $c(\alpha): j(\alpha)^*(\CA_I) \to j(\alpha)^* (\CA_{I_1} \boxtimes \CA_{I_2})$, and we say that the coalgebra $(\CA, c)$ is a \emph{factorization algebra} if each of these maps (and their analogues for compositions of the comultiplication map) is an isomorphism. 
\end{defn}

To give a weak version of this definition, we will work with appropriate open subsets of the Ran space of $X$. 

\begin{defn}\label{defn: suitable open sets}
Suppose given a family $\left\{W(I) \subset X^I \right\}_{I \in \fSet}$ of open subsets, each containing the diagonal $X$, such that for any surjection $\alpha : I \surj J$, $W(J)$ is the fibre product $W(I) \times_{X^I} X^J$. (We will denote the resulting map $W(I) \to W(J)$ by $\Delta_W(\alpha)$.) Assume in addition that for any $\alpha: I \surj J$, $W(I)$ is contained in $\prod_{j \in J} W(I_j)$. Such a family will be called a \emph{suitable} family of open subsets. 

We can form the colimit $W \defeq \colim_{I \in \fSet^\op} W(I)$ (with structure maps which we will denote by $\Delta^I_W$), and there is a natural morphism $\lambda: W \to \Ran{X}$ which is an open embedding of pseudo-indschemes. We say that any $W$ having such a presentation is \emph{suitable}.  
\end{defn}

\begin{rmk}
These will be the open subsets on which we require the data of weak factorization algebras to be given. Notice that this condition on the $W(I)$ is much stronger than required for the formulation of the definition of a weak factorization space---it is our belief that a weaker condition could also be used in this setting to define an alternate category of weak factorization algebras, which would still be equivalent to the category of ordinary factorization algebras; the objects which we are defining here should perhaps only be called ``slightly weak'' factorization algebras. However, we stop with this definition, since it is significantly easier to formulate, and entirely sufficient for our purposes: the regions $W(I)$ over which factorization data behaves well for the \'etale pullback of a factorization algebra form a suitable family, as we will see in example \ref{eg: suitable sets}.

Indeed, we expect that there is an a priori still more general notion, which we could call a ``very weak'' factorization algebra or space, over which the factorization data is only required to be given at each stage over a formal neighbourhood of the diagonal; futhermore we conjecture that all of these categories are equivalent, and hence ultimately it does not matter how strict we are in specifying our $W(I)$, since as in the proof of our theorem the data will always extend to all of $X^I$. However, this is beyond the scope of this paper.
\end{rmk}

\begin{eg}\label{eg: suitable sets}
Let $\phi: X \to Y$ be an \'etale morphism, and define subschemes $V_\phi(I) \subset X^I$ by 
\begin{align} \label{eq: suitable sets for pullbacks}
V_\phi(I) \defeq \left\{ x^I \ \vert \ \phi(x^{i_1}) = \phi(x^{i_2}) \text{ if and only if } x^{i_1} = x^{i_2} \right\}.
\end{align}
For fixed $I$, $V_\phi(I)$ is the intersection of the sets
\begin{align*}
V^{I,i,j}_\phi \defeq \left\{ x^I \in X^I \ \vert \phi(x^i) = \phi(x^j) \Leftrightarrow x^i = x^j \right\},
\end{align*}
as $i,j$ run over all unordered pairs in $I$. Arguing as in section \ref{sec: preliminary definition}, we see that the fact that $\phi$ is \'etale implies that $V^{I,i,j}_\phi$ is open in $X^I$, and hence so is $V_\phi(I)$. It is also clear that it contains the diagonal $\Delta(X)$. It is straightforward to check that the family $\left\{V_\phi(I)\right\}$ satisfies the conditions above, and hence give a suitable pseudo-indscheme.  

Furthermore, it is easy to check that if $\left\{W(I) \subset Y^I\right\}$ is a family of subschemes which is suitable with respect to $Y$, then we can define $\phi^*W(I) \defeq V_\phi(I) \cap (\phi^I)^{-1} W(I)$, and we obtain a family which is suitable with respect to $X$. This fact will be used when we define the \'etale pullback of a weak factorization algebra.  
\end{eg}

We will use implicitly throughout the following the identifications $\CD(W) \simeq  \lim_{I \in \fSet^\op} \CD(W(I)) \simeq \colim_{I \in \fSet^\op} \CD(W(I))$. We claim that this category $\CD(W)$ has a symmetric monoidal structure which is compatible with the chiral monoidal structure of $\CD(\Ran{X})$ under the restriction $\lambda^* : \CD(\Ran{X}) \to \CD(W)$; for this reason we will again denote the tensor product by $\otimes^\ch$. As in \cite{FG}, to define the monoidal structure it is enough to define for each $ J \in \fSet$ a natural transformation between the two functors
\begin{align*}
(\fSet^\op) ^{\times J} &\to \infty\text{-Cat} \\
(I_j)_{j \in J} &\mapsto \otimes_{j \in J} \CD(W(I_j)); \\
(I_j)_{j \in J} &\mapsto \CD\left(W\left(\bigsqcup_{j\in J} I_j\right)\right).
\end{align*}
This natural transformation is given for fixed $(I_j)_{j \in J}$ by the functor $\tau^W_{\alpha}$ which takes a collection of $\CD$-modules $\{M^{I_j} \in \CD(W(I_j))\}$ to the restriction of $j(\alpha)_* j(\alpha)^* \left( \boxtimes_{j \in K} M^{I_j} \right)$ from $\prod_{j \in J} W(I_j)$ to $W\left(\bigsqcup_{j \in J} I\right)$. (Here there is a slight abuse of notation: we use $j(\alpha)^*$ to denote the restriction from $\prod_{j \in J} W(I_j)$ to $U(\alpha) \cap W(I)$, and $j(\alpha)_*$ to denote the pushforward from $U(\alpha) \cap W(I)$ to $W(I)$.) 

It is straightforward to check from the definitions that $\lambda^*$ respects the chiral monoidal structures of $\CD(\Ran{X})$ and $\CD(W)$. Indeed, it suffices to observe that for any $\alpha: I \surj J$, the morphisms defining $\otimes^\ch$ on $\CD(\Ran{X})$ and $\CD(W)$, together with the morphisms defining $\lambda^*$, fit into a commutative diagram as follows:

\begin{center}
\begin{tikzpicture}[>=angle 90, bij/.style={above,sloped, inner sep=0.5pt}]
\matrix(a)[matrix of math nodes, row sep=2em, column sep=3em, text height=1.5ex, text depth=0.25ex]
{ \bigotimes_{j \in J} \CD(X^{I_j}) & \CD(X^I)\\
\bigotimes_{j \in J} \CD(W(I_j)) & \CD(W(I)). \\};
 \path[->, font=\scriptsize]
(a-1-1) edge node[above]{$\tau^X_\alpha$} (a-1-2)
(a-2-1) edge node[below]{$\tau^W_\alpha$} (a-2-2)
(a-1-1) edge node[left]{$\otimes \lambda_{I_j}^*$} (a-2-1)
(a-1-2) edge node[right]{$\otimes \lambda_{I}^*$} (a-2-2);
\end{tikzpicture}
\end{center}

Here the top horizontal arrow $\tau^X_\alpha$ is the functor used by Francis and Gaitsgory in defining the chiral tensor product on the Ran space, given on objects by 
\begin{align*}
\{M^{I_j} \in \CD(X^{I_j}) \} \mapsto j(\alpha)_*j(\alpha)^*(\boxtimes_{j} M^{I_j}).
\end{align*}
To check that the diagram is commutative, use the fact that the $\lambda_K^*$ are all restrictions along open embeddings. 

\begin{defn} A \emph{weak factorization algebra} on $X$ consists of a collection $\{W(I) \subset X^I \}$ as above, together with a cocommutative coalgebra object $(\CA, \tilde{c})$ of $(\CD(W), \otimes^\ch)$ such that for any $\alpha: I \surj J$, the corresponding maps $\tilde{c}_\alpha$ induced from $\tilde{c}$ are isomorphisms:
\begin{align*}
 j(\alpha)^* \CA_I \EquivTo j(\alpha)^* (\boxtimes_{j \in J} \CA_{I_j}) \in \CD(U(\alpha) \cap W(I)).
\end{align*} 
\end{defn}

\begin{rmk}
Let us take a moment to extract from the definition of a weak factorization algebra data which looks analogous to that comprising our definition of a weak factorization space. 
\begin{itemize}
\item The spaces $W(I)$ are given directly; likewise, we have for each $I$ a $\CD$-module $\CA_I = (\Delta_W^I)^! \CA$, which is the linear analogue of the space $\CZ_I \to W(I)$. 
\item The spaces $F(\alpha)$ are as large as possible, namely the intersection of $W(J)$ and $\Delta(\alpha)^{-1} W(I)$---but by the first compatibility condition on the $W$, this is exactly $W(J)$. The maps $\tilde{\nu}_\alpha: \CA_J \EquivTo \Delta_W(\alpha)^! \CA_I$ come from the structure of $\CA$ as a $\CD$-module on the pseudo-indscheme $W$. Likewise, this structure induces the higher compatibility maps between the $\tilde{\nu}$. 
\item The spaces $R(\alpha)$ are also as large as possible, a priori the intersection of $W(I)$ and $\prod_{j \in J} W(I_j)$, which is of course all of $W(I)$. The maps $\tilde{d}_\alpha$ are the inverses of the $\tilde{c}_\alpha$ from the definition of a weak factorization algebra. The higher compatibility morphisms between the $\tilde{d}_\alpha$ come from the coassociativity of the coproduct. (In particular, the domains $R$ on which they are defined are always as large as possible.)
\item The higher compatibilities between $\tilde{\nu}$ and $\tilde{d}$ are constructed from the structure of $\tilde{c}$ as a morphism in $\CD(W)$. Let us illustrate with the example of the first level of compatibility data: suppose that $\gamma: I \surj K$ factors through $J$ as $\beta \circ \alpha$. The restriction $(\Delta^I_W)^!\tilde{c}^K$ of the $K$th iterated comultiplication map gives rise to the factorization isomorphism $\tilde{c}_\gamma: j(\gamma)^* \CA_I \EquivTo j(\gamma)^* \left( \boxtimes_{k \in K} \CA_{I_k} \right)$ over $F(\gamma) \cap U(\gamma)$. Now the fact that $\tilde{c}^K$ is a morphism of $\CD$-modules on $W = \colim_{I \in \fSet^\op} W(I)$ means that we must have 2-isomorphism between $\Delta_W(\alpha)^!(\Delta_W^I)^! \tilde{c}^K$ and $(\Delta_W^J)^!\tilde{c}^K$. It is easy to check that since $\gamma = \beta \circ \alpha$, the intersection of $X^J \subset X^I$ and $U(\gamma) \subset X^I$ is just $U(\beta) \subset X^J$, and therefore that $U(\beta) \cap W(J)$ is just the intersection of $U(\gamma) \cap W(I)$ with $X^J \subset X^I$; we let $\Delta(\alpha)_\vert$ denote the the inclusion of $U(\beta) \cap W(J)$ in $U(\gamma) \cap W(I)$. Then the 2-isomorphism mentioned above gives rise to a 2-isomorphism making the following diagram (defined over $U(\beta) \cap W(J)$) commute:
\begin{center}
\begin{tikzpicture}[>=angle 90, bij/.style={above,sloped, inner sep=0.5pt}]
\matrix(a)[matrix of math nodes, row sep=2em, column sep=3em, text height=1.5ex, text depth=0.25ex]
{ \Delta(\alpha)_\vert^!j(\gamma)^*\CA_I & \Delta(\alpha)_\vert^! j(\gamma)^* \left(\boxtimes_{k \in K} \CA_{I_k} \right) \\
 j(\beta)^* \Delta_W(\alpha)^! \CA_I         & j(\beta)^* \Delta_W(\alpha)^! \left(\boxtimes_{k \in K} \CA_{I_k} \right) \\
  & \\
 j(\beta)^* \CA_K                                      & j(\beta)^* \left(\boxtimes_{k \in K} \CA_{J_k} \right) \\};
 \path[->, font=\scriptsize]
(a-1-1) edge node[above]{$\Delta(\alpha)_\vert^! \tilde{c}_\gamma$} (a-1-2)
(a-2-1) edge node[left]{$j(\beta)^* \tilde{\nu}_{\alpha}$} (a-4-1)
(a-2-2) edge node[right]{$j(\beta)^* \left( \boxtimes_{k \in K} \tilde{\nu}_{\alpha_k}\right)$} (a-4-2)
(a-4-1) edge node[below]{$\tilde{c}_\beta$} (a-4-2);
\path[->, bij]
(a-1-1) edge node[inner sep=1pt]{$\sim$} (a-2-1)
(a-1-2) edge node[inner sep=1pt]{$\sim$} (a-2-2);
\end{tikzpicture}
\end{center}
(Here we use the decomposition of $I$ and $J$ into subsets indexed by $K$ to decompose $\alpha$ into $K$-many maps $\alpha_k : I_k \surj J_k$.) Upon inverting the $\tilde{c}$, the 2-morphism making the resulting diagram commute is the first compatibility morphism between the $\tilde{\nu}$ and the $\tilde{d}$. In particular, note that the open set over which it is defined is $U(\beta) \cap W(J)$, which is as large as possible.
\end{itemize}
Because of this procedure, we will often refer to a weak factorization algebra $(\CA, \tilde{c})$ in terms of its pieces $\CA_I$ over $W(I)$; we will write $\{\CA_I\}$, omitting the rest of the factorization algebra structure, when confusion will not result. 
\end{rmk}

Since $\lambda^* : \CD (\Ran{X}) \to \CD (W)$ is an exact monoidal functor, it induces a functor from the category of factorization algebras on $X$ to the category of weak factorization algebras on $X$,
\begin{align*}
\Weak: \FAlg{X} \to \WFAlg{X,W}.
\end{align*} 

Just as for factorization spaces, we have the following:

\begin{thm}\label{thm: weak factorization algebras}
This functor is an equivalence of categories.
\end{thm}

Morally, this follows via the same constructions that provide a quasi-inverse in the factorization space case. However, we must address the usual infinity-categorical subtleties. We will do this in a moment, but first let us explain a less hands-on proof using chiral Koszul duality and the chiral algebra perspective which is now available to us. We can consider the category of chiral Lie algebra objects in $\CD(W)$, as well as the full subcategory of those objects supported on $X \subset W$---we call such objects \emph{chiral algebras on $W$}, and denote the category by $\ChAlg{W}$. As will be explained in the proof of Proposition \ref{prop: comparing chiral and factorization pullback}, the chiral Koszul duality equivalence of \cite{FG} extends to the categories of $\CD$-modules on $W$, and we obtain a commutative diagram
\begin{center}
\begin{tikzpicture}[>=angle 90, bij/.style={above,sloped, inner sep=0.5pt}]
\matrix(b)[matrix of math nodes, row sep=2em, column sep=2em, text height=1.5ex, text depth=0.25ex]
{  \ChAlg{X} &\FAlg{X} \\
   \ChAlg{W} &\WFAlg{X,W}. \\};
\path[->, bij]
 (b-1-1) edge node[inner sep=1pt]{$\sim$} (b-1-2)
 (b-2-1) edge node[inner sep=1pt]{$\sim$} (b-2-2);
 \path[->, font=\scriptsize]
 (b-1-2) edge node[right]{$\Weak$} (b-2-2)
 (b-1-1) edge node[left]{$\lambda^*$} (b-2-1);
\end{tikzpicture}
\end{center}
But it is easy to see that the restriction functor $\lambda^*: \ChAlg{X} \to \ChAlg{W}$ is an equivalence with quasi-inverse given by pushforward; it follows that $\Weak$ is also an equivalence. 

Now, as promised, let us see how the gluing construction gives an explicit quasi-inverse to $\Weak$. Given for example an object $\CA$ of $\WFAlg{X,W}$, we proceed inductively on $\vert I \vert$ (and use implicitly the fact that for each $n$ we need only work with one set $I$ of size $n$), to construct the sheaf $\Glue(\CA)_I$ on all of $X^I$, as well as all Ran condition and factorization condition isomorphisms $\nu_\alpha$, $d_\alpha$ for all partitions $\alpha$ of $I$ (again, there are only finitely many such partitions), and higher compatibility isomorphisms, for subpartitions $I \surj J \surj K$. At each level of the compatibility structures, we begin with partitions of minimal degeneracy (i.e. $\vert J \vert = \vert I \vert - 1$) and proceed to the maximally degenerate case $I \surj \pt$. In particular, a given surjection from $I$ can be factored into a composition of at most $n$ maps, and so the compatibility conditions need to be defined and checked only up to level $n$. Note that the resulting object $\Glue(\CA)$ has the property that for any factorization algebra $\CB$ on $X$, any morphism on $W$ from the restriction of $\CB$ to $\CA$ extends uniquely to a morphism from $\CB$ to $\Glue(\CA)$. Likewise, for any $n \in \BBN$ and any $n$-morphism in $\WFAlg{X,W}$, one can again explicitly construct the image of this morphism under $\Glue$ using the same gluing procedure.

We claim that this gives us the quasi-inverse to $\Weak$. However, the above procedure only defines the action of the ``functor'' on an individual object or $n$-morphism. To show that this structure fits together to produce an honest (infinity) functor, we argue as follows. We consider the functor
\begin{align*}
\Lambda:  \text{Co-Com}^\ch(\CD(\Ran{X})) \to \text{Co-Com}^\ch(\CD(W))
\end{align*}
induced by restriction from $\Ran{X}$ to $\Ran{W}$. Although the categories of (weak) factorization algebras are not presentable $\infty$-categories, the categories of cocommutative coalgebra objects are, as explained in Remark 2.4.9 of \cite{FG}. Since $\Lambda$ preserves colimits, we can apply the Adjoint Functor Theorem (Corollary 5.5.2.9 of \cite{L-HTT}), and conclude that $\Lambda$ has a right adjoint, which we will call $R$. We will show that the restriction of $R$ to the full-subcategory of weak factorization algebras on $W$ lands in the category of factorization algebras on $X$, by showing that it agrees on objects with the definition of $\Glue$ given above. This completes the definition of $\Glue$ as a functor; it will then remain to check that this functor is indeed a quasi-inverse to $\Weak$ as claimed. 

First observe that for a given object $\CA$ of $ \text{Co-Com}^\ch(\CD(W))$, we can calculate the pieces $R\CA_{X^I}$ of $R\CA$ by hand, inductively: $R\CA_X$ is just $\CA_X$, and $R\CA_{X^I}$ is the pullback of $\lambda_{I,*}\CA_I$ and $\oplus_{\alpha: I \surj \{1,2\}} j(\alpha)_* j(\alpha)^* (R\CA_{X^{I_1}} \boxtimes R\CA_{X^{I_2}})$ over the direct sum of the corresponding terms involving $\CA_{I_1}, \CA_{I_2}$. In particular, it is straightforward to check that the composition $\Lambda \circ R$ is equivalent to the identity. 

Next, we claim that for any object $\CA$ of $\WFAlg{X,W}$, $R\CA$ agrees with $\Glue(\CA)$, which is by construction a factorization algebra. To prove this, we use the properties of the adjuntion $(\Lambda, R)$, and the fact that restricting back to $W$ after applying either $R$ or $\Glue$ yields the original object or morphism. More precisely, given an object $\CA$ of $\WFAlg{X,W}$, the identity map $\CA \to \CA$ corresponds under the adjunction to a morphism $\Phi: \Glue(\CA) \to R\CA$. Furthermore, the identity map $R\CA_{|W} \to \CA$ also results in a morphism $\Psi: R\CA \to \Glue(\CA)$, by construction of $\Glue(\CA)$ as described above. Now we look at the compositions $\Phi \circ \Psi: \Glue(\CA) \to \Glue(\CA)$ and $\Psi \circ \Phi: R\CA \to R\CA$. In both cases, the fact that they restrict to the identity over $W$ is enough (by construction of $\Glue$ in the first case, and by adjunction in the second) to show that each is equivalent to the identity, and hence $\Phi$ and $\Psi$ are quasi-inverse isomorphisms.

Thus we can consider $\Glue$ as a functor. The fact that $\Lambda \circ R$ is equivalent to the identity functor on $ \text{Co-Com}^\ch(\CD(W))$ implies that $\Weak \circ \Glue$ is equivalent to the identity functor on $\WFAlg{X,W}$. On the other hand, we have a natural morphism $\text{Id}_{\FAlg{X}} \to \Glue \circ \Weak$ arising from the adjunction $(\Lambda, R)$. It is easy to see that this is also an equivalence. We conclude that $\Glue$ and $\Weak$ are quasi-inverse isomorphisms, and thus that $\FAlg{X}$ is equivalent to $\WFAlg{X,W}$ for any suitable $W$, as claimed.  \qed

\section{\'Etale pullback of factorization spaces}
\label{sec: pullback of factorization spaces}
Let us assume that we have an \'etale morphism $\phi: X \to Y$ of smooth varieties, and let $\CZ= (\CZ_I, W(I), \ldots)$ be a weak factorization space on $Y$. In this section, our goal is to define a weak factorization space $\CZ^\prime = \phi^*\CZ$ over $X$. 

In the case $I = \{\pt\}$, we have $W({\{\pt\}}) = Y$. We set $W^\prime(\{\pt \}) = X$, and we define $\CZ^\prime_{\{\pt\}} = X \times_{Y} \CZ_{\{\pt\}}$. 

For more general $I$, recall from example \ref{eg: suitable sets} that the sets
\begin{align*}
V_\phi(I) \defeq \left\{ x^I \in X^I \ \vert \ \phi(x^{i_1}) = \phi(x^{i_2}) \Leftrightarrow x^{i_1} = x^{i_2} \text{ for } i_1,i_2 \in I \right\}.
\end{align*}
define a suitable family, and likewise so do the sets $W^\prime(I) \defeq \phi^* W(I)\subset X^I$ given by the intersection
\begin{align*}
(\phi^I)^{-1}(W(I)) \cap V_\phi(I) = V_\phi(I) \times_{Y^I} W(I).
\end{align*}
Now we let $\CZ_I^\prime$ be equal to the pullback
\begin{align*}
\CZ_I^\prime \defeq W(I)^\prime \times_{W(I)} \CZ_I.
\end{align*}
It is immediate that $\CZ^\prime_I$ is an indscheme with connection over $W^\prime(I)$. 

\begin{prop}\label{prop: etale pullback gives a weak factorization space}
There is a natural structure of weak factorization space on the data $\left\{\CZ_I^\prime, W^\prime(I)\right\}$, induced from the weak factorization structure on $\{\CZ_I, W(I)\}$. 
\end{prop}
\begin{proof}
First let us check that the weak version of Ran's condition is satisfied. Let $\alpha: I \surj J$ be a surjection of finite sets. We need to find an open neighbourhood $R^\prime(\alpha)$ of the diagonal $\Delta(X)$ in $X^J$ over which we can identify the restriction of $\CZ^\prime_J$ and the pullback of $\CZ^\prime_I$. We will use the fact that we have such an identification $\tilde{\nu}_\alpha$ of $\CZ_J$ and the pullback of $\CZ_I$ over the open set $R(\alpha)$; thus we can define the desired isomorphism over the intersection of $W^\prime(J)$ with the preimage of $R(\alpha)$ under $\phi^J$ and the preimage of $W^\prime(I)$ under the embedding $\Delta(\alpha): X^J \emb X^I$. That is, we take
\begin{align*}
R^\prime(\alpha) \defeq \left(W^\prime(I) \times_{X^I} W^\prime(J) \right) \times_{Y^J} R(\alpha).
\end{align*}
Then we have
\begin{align*}
R^\prime(\alpha) \times_{W^\prime(J)} \CZ^\prime_J \simeq (W^\prime(I) \times_{X^{I}} W^\prime(J) ) \times _{Y^J} \left(R(\alpha) \times_{W(J)} \CZ_J \right),
\end{align*}
while
\begin{align*}
R^\prime(\alpha) \times_{W^\prime(I)} \CZ_I^\prime \simeq (W^\prime(I) \times_{X^{I}} W^\prime(J) ) \times _{Y^J} \left(R(\alpha) \times_{W(I)} \CZ_I \right).
\end{align*}

From this presentation, it is clear that we should define $\tilde{\nu}^\prime_\alpha$ to be 
\begin{align*}
\id_{W^\prime(I) \times_{X^I} W^\prime(J)} \times \tilde{\nu}_\alpha
\end{align*}
(composed with the natural isomorphisms above). 

Let us next define the factorization isomorphisms $\tilde{d}^\prime_\alpha$. We set
\begin{align*}
F^\prime(\alpha) &\defeq \left( W^\prime(I) \times_{X^I} \left( \prod_{j \in J} W^\prime(I_j) \right) \right) \times_{Y^I} F(\alpha) 
\\ &= W^\prime(I) \cap \left(\prod_{j \in J} W^\prime(I_j) \right) \cap  (\phi^I)^{-1} \left(F(\alpha) \right).
\end{align*}

Then we have that $F^\prime(\alpha) \times_{W^\prime(I)} \CZ^\prime_I$ is canonically isomorphic to
\begin{align*}
\left( W^\prime(I) \times_{X^I} \left( \prod_{j \in J} W^\prime(I_j) \right) \right) \times_{Y^I} \left( (F(\alpha) \times _{W(I)} \left( \prod_{j \in J} \CZ_{I_j} \right) \right),
\end{align*}
and similarly $F^\prime(\alpha) \times _{\prod_{j \in J} W^\prime(I_j)} \left(\prod_{j \in J} \CZ^\prime_{I_j} \right)$ can be identified canonically with
\begin{align*}
\left( W^\prime(I) \times_{X^I} \left( \prod_{j \in J} W^\prime(I_j) \right) \right) \times_{Y^I} \left( (F(\alpha) \times_{\prod_{j \in J} W(I_j)} \left( \prod_{j \in J} \CZ_{I_j} \right) \right). 
\end{align*}
Hence we can take $\tilde{d}^\prime_\alpha$ to be the isomorphism
\begin{align*}
\id_{ W^\prime(I) \times_{X^I} \left( \prod_{j \in J} W^\prime(I_j) \right)} \times \tilde{d}_\alpha.
\end{align*}
It is clear that $\tilde{d}^\prime_\alpha$ and $\tilde{\nu}^\prime_\alpha$ satisfy the required compatibilities, because $\tilde{d}_\alpha$ and $\tilde{\nu}_\alpha$ do. Therefore, they give the weak factorization structure on $\{\CZ^\prime, W^\prime(I)\}$, and the proof is complete. 
\end{proof}

Combining Proposition \ref{prop: etale pullback gives a weak factorization space} and Theorem \ref{thm: weak factorization spaces are factorization spaces} we can understand how to pull back a factorization space along an \'etale morphism $\phi$. We define a weak factorization space in the na\"ive way over the open sets $V^I_{\phi} \subset X^I$; then we use the functor $\Glue$ to extend the components of the weak factorization space to all of $X^I$. 

Similarly, we can define the pullback of a factorization algebra along an \'etale morphism. In the language of categories of $\CD$-modules on the Ran space and its suitable subspaces $W$, we can express this result as follows. Suppose that $\{W(I) \subset Y^I \}$ is a suitable family over $Y$, and $\{W^\prime(I) = \phi^*W(I)\subset X^I \}$ is the suitable family over $X$ defined using an \'etale morphism $\phi: X \to Y$ as in example \ref{eg: suitable sets}. Let $W$ and $W^\prime$ be the corresponding pseudo-indschemes, with $\overline{\phi}$ the obvious morphism $W^\prime \to W$. 

Then pullback along $\overline{\phi}$ is monoidal with respect to the chiral monoidal structures on $\CD(W)$ and $\CD(W^\prime)$, and hence takes coalgebra objects to coalgebra objects. Furthermore, it takes weak factorization algebras to weak factorization algebras. It is straightforward to check that this pullback functor agrees with the pullback defined more explicitly for each piece over $W(I)$ as in the case of weak factorization spaces above. 

Note that the pullback morphism from $\CD(\Ran{Y})$ to $\CD(\Ran{X})$ fails to be monoidal in general, for exactly the same reason that the na\"ive pullback of a factorization space or factorization algebra fails to satisfy factorization. This is why it is necessary to pass through the category of weak factorization algebras. 

\section{Examples of universal factorization spaces}
\label{sec: examples}
Our goal is now to give a precise definition of a universal factorization space in some dimension $d$. Roughly, it should be an assignment of a factorization space to each smooth $d$-dimensional variety, in a way behaving well with respect to \'etale morphisms between varieties, but also behaving well in families. In order to carefully formulate this condition, we need to define the notion of a family of factorization spaces. 

\begin{defn}
Let $S$ be a scheme of finite type over $k$, and let $\pi: X \to S$ be a smooth morphism of dimension $d$. Let $(X/S)^I$ denote the $I$-fold fibre product 
\begin{align*}
X \times_S X \times_S \ldots \times_S X \simeq X^I \times_{S^I} S.
\end{align*}
A \emph{relative factorization space over $X/S$} consists of the following data:

\begin{enumerate}
\item For each finite set $I$, we have a prestack $\CY_{(X/S)^I} \in \operatorname{PreStk}_{/S}$, representable by an indscheme, and equipped with a map
\begin{align*}
f_I: \CY_{(X/S)^I} \to (X/S)^I 
\end{align*}
and a formally integrable relative connection over $(X/S)^I/S$. 
\item For each surjection $\alpha: I \surj J$, we require an isomorphism $\nu_{\alpha}$ fulfilling Ran's condition over the diagonal $(X/S)^J \emb (X/S)^I$. 
\item For each surjection $\alpha: I \surj J$, we also require an isomorphism $d_{\alpha}$ fulfilling the factorization condition over $U(\alpha) \times_{X^I} (X/S)^I$.
\end{enumerate}
We require that these isomorphisms be compatible with each other and with composition. 
\end{defn}

\begin{rmk}
\begin{enumerate}
\item Note that this is strictly weaker than a factorization space over the total space $X$: not only are the spaces only required to be defined and to be equipped with the appropriate isomorphisms over smaller spaces $(X/S)^I \subset X^I$, the connection is only required to be defined along the fibres of $X$ over $S$. 
\item Note also that this is \emph{not} an example of a weak factorization space over $X$, because $(X/S)^I$ need not contain an open neighbourhood of the diagonal $\Delta(X)$. 
\item On the other hand, given a factorization space on the total space $X$, restriction of each piece to the appropriate $(X/S)^I$ does give a relative factorization space. 
\end{enumerate}
\end{rmk}

We can now formulate the notion of a universal factorization space; it is modelled on the definition of a universal $\CD$-module as in \cite{BD1}, 2.9.9:

\begin{defn}
\label{defn: universal factorization space}
Let $d$ be a positive integer. A \emph{universal factorization space of dimension $d$} consists of the following data:
\begin{enumerate}
\item For each smooth family $\pi: X \to S$ of relative dimension $d$, we require a relative factorization space
\begin{align*}
\CY_{\Ran(X/S)} = \left\{ \CY_{(X/S)^I} \to (X/S)^I \right\}
\end{align*}
over $X/S$.
\item For each fibrewise \'etale morphism of smooth $d$-dimensional families $\phi: X/S \to X^\prime/S^\prime$ we require an isomorphism
\begin{align*}
\CY(\phi) : \CY_{\Ran(X/S)} \EquivTo \phi^* \CY_{\Ran(X^\prime/S^\prime)}
\end{align*}
of relative factorization spaces. These isomorphisms are required to be compatible with composition of fibrewise \'etale morphisms. 
\end{enumerate}
\end{defn}

Let us now discuss some important examples of factorization spaces from the literature, and check that they are indeed compatible with respect to pullback along \'etale morphisms between smooth curves. (It is also possible to define relative versions of these factorization spaces, over families of smooth curves, but since the focus of this paper is on the pullback, we will only discuss this property.) Our first two examples will live over curves, and will be the Beilinson--Drinfeld Grassmannian and the space of meromorphic jets of \cite{KV}. 

As a third example, we will sketch the construction of a factorization space over any smooth variety $Y$ using the Hilbert scheme of points of $Y$. The details of this construction are contained in the author's thesis, and will appear in a future paper; we include the sketch here only to illustrate that universal factorization spaces do exist over higher-dimensional varieties. 

The critical observation in all three examples is the following: let $\phi: Y \to Y^\prime$ be an \'etale morphism; let $S$ be an arbitrary scheme, and consider a morphism $y^I: S \to Y^I$ whose image lies in the open subscheme $V_\phi(I)$ of $Y^I$ defined as in (\ref{eq: suitable sets for pullbacks}); and let $\{y^I\}$ denote the union of the graphs of the functions $y^i: S \to Y$. Then the \'etale morphism $\phi$ induces an isomorphism between the formal schemes corresponding to the formal completions of the graphs in $S \times Y$:
\begin{align} \label{eq: isomorphism of formal neighbourhoods}
(S \times Y)^{\wedge}_{\graph{y^I}} \EquivTo (S \times Y^\prime)^{\wedge}_{\{\phi^I \circ y^I \}}.
\end{align}

\begin{eg}\label{eg: grassmannian}
Let $G$ be a reductive algebraic group, let $C$ be a smooth curve, and recall the \emph{Beilinson--Drinfeld Grassmannian}, introduced in \cite{BD2} and defined by the family
\begin{align*}
\GrGXI{C}: S \mapsto \GrGXI{C}(S) \defeq \left\{ (c^I, \CP, \sigma) \ \left| \begin{array}{l}
c^I: S \to C^I; \\
\CP \to S \times C \text{ a principal $G$-bundle; }\\
\sigma: S \times C \setminus \{c^I\} \to \CP \text{ a section}
\end{array}
\right. \right\}.
\end{align*}
\end{eg}

\begin{prop}
The Beilinson--Drinfeld Grassmannian is universal with respect to pullback along \'etale morphisms between the smooth curves $\emph{C}$. 
\end{prop}

\begin{proof}
Let $\phi: C \to D$ be an \'etale morphism of smooth curves. We wish to show that the pullback of the Beilinson--Drinfeld Grassmannian over $D$ along $\phi$ is isomorphic to the Beilinson--Drinfeld Grassmannian over $C$; in order to do this it is sufficient to show that for each finite set $I$ we have isomorphisms of the $I$th components over the open subscheme $V_\phi(I) \subset C^I$. By definition, the $I$th component of the pullback on $V_\phi(I)$ is 
\begin{align*}
V_\phi(I) \times_{D^I} \GrGXI{D}.
\end{align*}
An $S$-point of this space is given by the data $(c^I, d^I, \CP, \sigma)$, where $c^I: S \to V_\phi(I) \subset C^I$, $d^I = \phi^I \circ c^I$, $\CP$ is a principal $G$-bundle on $S \times D$, and $\sigma$ is a trivialization of $\CP$ away from $\{ d^I \}$. 

Recall that the data of $\CP$ and $\sigma$ is equivalent to the data of a principal $G$-bundle on the formal neighbourhood of the graph $\{d^I\}$ together with a trivialization of this bundle on the punctured formal neighbourhood. But since the formal neighbourhood of $\{d^I\}$ in $S \times D$ is isomorphic to the formal neighbourhood of $\{c^I\}$ in $ S \times C$ as in (\ref{eq: isomorphism of formal neighbourhoods}), this data is equivalent to the data of a principal $G$-bundle $\CP^\prime$ on $S \times C$ together with a trivialization $\sigma^\prime$ away from the graph $\{c^I\}$. That is, we have a canonical isomorphism
\begin{align*}
\left(V_\phi(I) \times_{D^I} \GrGXI{D} \right)(S) \simeq \left( V_\phi(I) \times_{C^I} \GrGXI{C} \right)(S).
\end{align*}
It is straightforward to see that these isomorphisms are functorial in $S$, and hence induce an isomorphism
\begin{align*}
V_\phi(I) \times_{D^I} \GrGXI{D} \to V_\phi(I) \times_{C^I} \GrGXI{C}
\end{align*}
for each finite set $I$. Moreover, as we allow $I$ to vary, the resulting isomorphisms are compatible with the factorization structures, and hence provide an isomorphism of weak factorization spaces. Finally, we conclude by Theorem \ref{thm: weak factorization spaces are factorization spaces} that we have an isomorphism of factorization spaces as desired.
\end{proof}

\begin{eg} \label{eg: meromorphic jets KV}
Let us now study the factorization space of meromorphic jets, defined in (3.3.2) of \cite{KV}. Fix $X=\Spec(A)$ an affine scheme, $C$ a smooth curve, and $I$ a finite set. We are interested in the functor
\begin{align*}
\CL(X)_{C^I}: S \to \CL(X)_{C^I}(S) \defeq \{ (c^I, \rho) \},
\end{align*}
where $c^I: S \to C^I$ is a morphism, and $\rho$ is a meromorphic function on the formal neighbourhood of the graph $\{c^I\}$ in $S \times C$. More formally, $\rho$ is a morphism of $k$-algebras
\begin{align*}
A \to \CK_{c^I},
\end{align*}
where $\CK_{c^I}$ is the ring of functions on the punctured formal neighbourhood of $\{c^I\}$. 

Then it is clear that the isomorphism (\ref{eq: isomorphism of formal neighbourhoods}) induces an isomorphism
\begin{align*}
V_\phi(I) \times_{D^I} \CL(X)_{D^I} \EquivTo V_\phi(I) \times_{C^I} \CL(X)_{C^I},
\end{align*}
giving rise to an isomorphism of the factorization space $\CL(X)_{\Ran{C}}$ with the pullback of the factorization space $\CL(X)_{\Ran{D}}$. 
\end{eg}

\begin{eg} \label{eg: Hilbert scheme}
We now sketch the construction of a new factorization space $Y$ over any smooth variety $Y$. To our knowledge, this is the first construction of a non-trivial factorization space over a variety of dimension greater than one. A very similar space was constructed independently (but never written down) by Masoud Kamgarpour and Anthony Licata. 

Let $\Hilb^n(Y)$ denote the Hilbert scheme of $n$ points in $Y$, and let $\Hilb(Y)$ denote the disjoint union over all Hilbert schemes for $n\ge 0$. Our goal is to use the Hilbert scheme of $Y$ to build a factorization space $\{ \HilbI{Y}{I} \}$ over $Y$. We begin by defining the functor of points of each space $\HilbI{Y}{I}$ for $I \in \fSet$ a finite set. 

\begin{defn}\label{def: Hilbert factorization space}
Given $I \in \fSet$, let $\HilbI{Y}{I}$ be the prestack sending a test scheme $S$ to the set
\begin{align*}
\left\{ (\xi,y^I) \ \left|
\begin{array}{l}
y^I: S \to Y^I; \\
\xi \in \Hilb(Y)(S);\\
\Supp(\xi) \subset \left\{y^I\right\} \text{ set-theoretically}
\end{array}
\right.
\right\}.
\end{align*}
\end{defn}

\begin{rmk} \label{rmk: set-theoretic condition}
Here $\left\{y^I\right\} \subset S \times Y$ is the union over $I$ of the graphs of each of the functions $y^i: S \to Y$. The condition on the support of $\xi \in \Hilb^n(Y)(S)$ can be interpreted as follows: recall that $\xi$ gives rise to $n$ unordered morphisms $\xi_j: S \to Y$; then for each $j$ the graph of $\xi_j$ must be equal set-theoretically to the graph of some $y^i: S \to Y$. Equivalently, we must have $\xi_j \circ \redEmb{S} = x^i \circ \redEmb{S} : S_\red \to Y$, where $\redEmb{S}: S_\red \emb S$ is the canonical inclusion of the reduced scheme $S_\red$ in $S$.
\end{rmk}
\end{eg}

We omit the proof that these functors are representable by indschemes: the strategy is to show that each  $\HilbI{Y}{I}$ is representable by an ind-closed sub-indscheme of $Y^I \times \Hilb(Y)$, equipped with an integrable connection over $Y^I$. The proof is technical but not deep. It is straightforward to verify that these spaces satisfy Ran's condition and the factorization condition, and hence give rise to a factorization space over $Y^I$. The isomorphism (\ref{eq: isomorphism of formal neighbourhoods}) again allows us to show that the spaces are compatible under \'etale pullback.

The detailed proofs of these facts are beyond the scope of this paper, and will be included in a separate paper. 

\section{Pullback of factorization and chiral algebras}
\label{sec: pullback of factorization and chiral algebras}
Let $\phi: X \to Y$ be an \'etale morphism. We have defined a functor 
\begin{align*}
\phi^*: \FAlg{Y} \to \FAlg{X}.
\end{align*}
In this section, we check that it is compatible with the pullback functor
\begin{align*}
\phi^*_{\ch}: \ChAlg{Y} \to \ChAlg{X}
\end{align*}
under Koszul duality.

Let us begin by recalling the definition of the chiral \'etale pullback functor $\phi^*_{\ch}$. We will do this in some detail, since it is not written elsewhere in the literature. Let $(\CB_Y, \mu_Y)$ be a chiral algebra on $Y$: $\CB_Y$ is a $\CD$-module on $Y$, and $\mu_Y$ is a Lie bracket on the pushforward $\CB_{\Ran{Y}}$ of $\CB_Y$ to the Ran space:
\begin{align*}
\mu_Y: \CB_{\Ran{Y}} \otimes^\ch \CB_{\Ran{Y}} \to \CB_{\Ran{Y}}.
\end{align*}

For our purposes we will concentrate on the restriction of this map to $Y^2$; for repeated copies of $Y$ the argument is similar. Let us denote this restricted map also by $\mu_Y$:
\begin{align*}
\mu_Y: (j_Y)_*j_Y^*(\CB_Y \boxtimes \CB_Y) \to (\Delta_Y)_! \CB_Y.
\end{align*}

(Here and throughout this section, $\Delta_Y$ and $j_Y$ are the inclusions of the diagonal and its complement, which will be denoted $U_Y$, in $Y^2$.) The pullback of the chiral algebra $(\CB_Y, \mu_Y)$ has as underlying $\CD_X$-module simply the module $\CB_X \defeq \phi^! (\CB_Y)$. The chiral bracket $\mu_X$ is defined on $X^2$ in the following way. Recall the open subscheme $V_\phi(\{1,2\})$ defined in \ref{eg: suitable sets}; we will denote it by $V_\phi(2)$, and will let $j_V$ denote its open embedding in $X^2$. Since $V_\phi(2)$ and $U_X$ given an open cover of $X^2$, it is sufficient to define the map $\mu_X$ on each of these pieces in a compatible way. 

Moreover, notice that the restriction of $(\Delta_X)_! (\CB_X)$ to $U_X$ is zero, and so the restriction of $\mu_X$ to $U_X$ is of course zero. It follows that it suffices to define the map $\mu_X$ on $V_\phi(2)$. 

Next notice that $j_V^* (j_X)_* (j_X)^* (\CB_X \boxtimes \CB_X)$ is canonically isomorphic to 
\begin{align*}
j_V^* (\phi^2)^! (j_Y)_* j_Y^* (\CB_Y \boxtimes \CB_Y).
\end{align*}
Indeed, this can be seen from the following diagram of distinguished triangles of sheaves on $V_\phi(2)$:

\begin{center}
\begin{tikzpicture}[>=angle 90]
\matrix(a)[matrix of math nodes, row sep=3em, column sep=1em, text height=1.5ex, text depth=0.25ex]
{ \to j_V^* (\CB_X \boxtimes \CB_X) & j_V^* (j_X)_* j_X^* (\CB_X \boxtimes \CB_X) & j_V^* (\Delta_X)_! (\Delta_X)^!(\CB_X \boxtimes \CB_X) \to\\
  \to j_V^* (\phi^2)^! (\CB_Y \boxtimes \CB_Y) & j_V^* (\phi^2)^! (j_Y)_* j_Y^* (\CB_Y \boxtimes \CB_Y) & j_V^* (\phi^2)^! (\Delta_Y)_! (\Delta_Y)^! (\CB_Y \boxtimes \CB_Y) \to \\};
  
\path[->, font=\scriptsize]
 (a-1-1) edge (a-1-2) 
 (a-1-2) edge (a-1-3)
 (a-2-1) edge (a-2-2)
 (a-2-2) edge (a-2-3)
 (a-1-1) edge (a-2-1)
 (a-1-3) edge (a-2-3)
 (a-1-2) edge[dashed] (a-2-2); 
\end{tikzpicture}
\end{center}
(Here the left vertical isomorphism follows from the definition of $\CB_X$ and the compatibility of the exterior tensor product with pullback, and the right vertical map is a base change isomorphism.)

Furthermore, base change also gives an isomorphism
\begin{align*}
j_V^* (\phi^2)^! (\Delta_Y)_! \CB_Y \EquivTo j_V^* (\Delta_X)_! \CB_X.
\end{align*}

From this we see that we can define $j_V^*(\mu_X)$ to be the pullback of $\mu_Y$:
\begin{align*}
j_V^* (\phi^2)! (j_Y)_* j_Y^* (\CB_Y \boxtimes \CB_Y) \xrightarrow{j_V^* (\phi^2)^* (\mu_Y)} j_V^* (\phi^2)^! (\Delta_Y)_! \CB_Y.
\end{align*}

Similarly, the higher compatibility morphisms need only be defined over the sets $V_\phi(I)$, where they can be defined as the pullback of the higher compatibility morphisms for $\mu_Y$. 

We can rephrase this in terms of the Ran space description of chiral algebras. The suitable family $\{V_\phi(I)\}$ defines $V_\phi \subset \Ran X$; the functor given by pulling back along the composition
\begin{align*}
V_\phi \emb \Ran X \to \Ran Y
\end{align*}
respects the chiral monoidal structures on $\CD(\Ran Y)$ and $\CD(V_\phi)$, and hence takes Lie algebra objects to Lie algebra objects. Furthermore, the condition that the $\CD_{\Ran Y}$-module $\CB$ underlying a chiral algebra be the pushforward of a $\CD_Y$-module $\CB_Y$ implies that the corresponding $\CD_{V_\phi}$-module is the pushforward of $\phi^*\CB_Y$ along the canonical map $X = V_\phi(\{*\}) \to V_\phi$. It is straightforward to check that a Lie algebra object in $\CD(V_\phi)$ with this property is equivalent to a chiral algebra on $X$: one extends the underlying $\CD$-module by pushing forward along $X \to \Ran X$, and all of the structure morphisms are easily seen to be trivial away from the diagonals, and hence away from $V_\phi$. Restriction from $\Ran X$ to $V_\phi$ gives the quasi-inverse. 

\begin{prop}
\label{prop: comparing chiral and factorization pullback}
Let $\{\CA_{Y^I}\}$ be a factorization algebra over $Y$, and let $(\CB_Y, \mu_Y)$ denote the corresponding chiral algebra on $Y$. Suppose that $\phi: X \to Y$ is an \'etale morphism, and let $\{\CA_{X^I}\}$ denote the pullback of the factorization algebra to $X$. Then the chiral algebra $(\CB^\prime_X, \mu_X^\prime)$ associated to this factorization algebra is canonically isomorphic to the chiral algebra $\CB_X \defeq \phi^*_\ch  \CB_Y$.
\end{prop}
\begin{proof}
Let us first give an abstract proof in the language of $\CD$-modules on the Ran spaces. Then we will give an explicit proof in terms of the underlying modules and maps on products of $X$ and $Y$. As described above, we have a monoidal functor of chiral monoidal categories
\begin{align*}
\overline{\phi}^*: \CD(\Ran Y) \to \CD(V_\phi).
\end{align*}
Following Francis--Gaitsgory, Koszul duality gives an equivalence between categories of Lie algebra objects and categories of cocommutative Lie algebra objects in each of these monoidal categories. (A modification of the proof that $\CD(\Ran Y)$ is pro-nilpotent shows that $\CD(V_\phi)$ is pro-nilpotent as well, and then Proposition 4.14 of \cite{FG} gives the equivalence.) The Koszul duality functors are functorial; in particular, they are intertwined by $\overline{\phi}^*$, as well as by the restriction functor $\lambda^*$ from $\Ran X$ to $V_\phi$. Theorem 5.2.1 of \cite{FG} shows that Koszul duality restricts to equivalences between the categories of chiral algebras and factorization algebras, and from our discussions above, we know that $\overline{\phi}^*$ and $\lambda^*$ also preserve these categories. This gives the desired result. 

However, in applications it may be useful to have a more hands-on comparison of the actual brackets and sheaves involved, so we will provide that here. Let us begin by recalling the construction of the chiral algebra $(\CB_Y, \mu_Y)$ from the factorization algebra $\{\CA_{Y^I}\}$, in particular focusing on the higher compatibility isomorphisms, which are not discussed elsewhere in the literature. 

The $\CD_Y$-module underlying the chiral algebra is $\CB_Y \defeq \CA_{Y}[-1]$. 

The chiral bracket $\mu_Y$ is given by composing the maps from $(j_Y)_* j_Y^* \left(\CA_Y[-1] \boxtimes \CA_Y[-1] \right)$ to $(\Delta_Y)_! \left( \CA_Y [-1] \right)$ in the following diagram:

\begin{center}
\begin{equation} \label{diagram: chiral bracket}
\begin{tikzpicture}[baseline=(current  bounding  box.center), >=angle 90,  bij/.style={above,sloped,inner sep=0.5pt}]
\matrix(b)[matrix of math nodes, row sep=2em, column sep=3em, text height=1.5ex, text depth=0.25ex, ampersand replacement=\&]
{          \& (j_Y)_* j_Y^* \left(\CA_Y[-1] \boxtimes \CA_Y[-1] \right) \& \\
           \& (j_Y)_* j_Y^* \left( \left(\CA_Y \boxtimes \CA_Y\right)[-2]\right) \& \\
 \CA_{Y^2}[-2] \& (j_Y)_* j_Y^* \left(\CA_{Y^2}[-2] \right)             \& (\Delta_Y)_! (\Delta_Y)^! \left(\CA_{Y^2}[-1] \right)\\
           \&                                             \& (\Delta_Y)_! \left(\CA_Y[-1]\right). \\};
\path[->, font=\scriptsize]
(b-1-2) edge node[bij]{$\sim$} node[left]{$(j_Y)_* j_Y^* \tau_{\CA_Y}$} (b-2-2)
(b-2-2) edge node[bij]{$\sim$} node[left]{$(j_Y)_*d_Y$} (b-3-2);
\path[->, font=\scriptsize]
(b-3-1) edge (b-3-2)
(b-3-2) edge node[above]{Res} (b-3-3);
\path[->, font=\scriptsize]
(b-3-3) edge node[bij]{$\sim$} node[left]{$(\Delta_Y)_!\nu_Y$} (b-4-3);
\end{tikzpicture}
\end{equation}
\end{center}

Here $\tau_{\CA_Y}$ is the supersymmetric identification of the complexes of $\CD_{Y^2}$ modules
\begin{align*}
\tau_{\CA_Y}: \CA_Y[-1] \boxtimes \CA_Y[-1] \EquivTo \left(\CA_Y \boxtimes \CA_Y\right)[-2].
\end{align*}
The higher isomorphisms giving the graded skew-symmetry of $\mu_Y$ are induced by those of $\tau_{\CA_Y}$ (composed with those describing the symmetry of the other maps in the composition). 

Finally, the higher isomorphisms comprising the Jacobi identity come from the Cousin complexes of the $\CD_{Y^I}$-module $\CA_{Y^I}$ with respect to the diagonal stratification. We will focus on the case of $\CA_{Y^3}$, to see the first levels of the structure. In this case, we introduce notation which is hopefully self-explanatory: for example, $j_{(1\ne 2,3)}$ is the inclusion of the open subscheme $\{(y_1, y_2, y_3) \ \vert \ y_1 \ne y_2, y_3 \}$ in $Y^3$. In our earlier notation, this $j_{1\ne 2, 3}$ would have been written as $j_Y(\alpha)$, for $\alpha: \{ 1,2,3 \} \surj \{1, \{2,3\}\}$. We also denote by $\Delta_{(3)}$ the inclusion of the small diagonal,  and by $j_{(3)}$ the inclusion of the complement of the fat diagonal. Then the complex is as follows:
\begin{small}
\begin{center}
\begin{tikzpicture}[>=angle 90,  bij/.style={above,sloped,inner sep=0.5pt}]
\matrix(b)[matrix of math nodes, row sep=2em, column sep=3em, text height=1.5ex, text depth=0.25ex]
{ j_{(3),*}j_{(3)}^* \CA_{Y^3} & \displaystyle{\bigoplus_{i \in {1,2,3}} j_{(i \ne j,k),*} j_{(i\ne j, k)}^* \Delta_{(j=k),!} \Delta _{(j=k)}^! \left(\CA_{Y^3}[1] \right)} & \Delta_{(3),!} \Delta_{(3)}^! \left(\CA_{Y^3}[2] \right) \\};
\path[->, font=\scriptsize]
(b-1-1) edge node[above]{$d$} (b-1-2)
(b-1-2) edge node[above]{$d$} (b-1-3);
\end{tikzpicture}
\end{center}
\end{small}
Here the sum is over $i = 1,2,3$, and $(j, k)$ is the unordered pair consisting of the remaining two elements, so that there are three summands in the middle term. The components of the differentials $d$ are given by suitable residue maps; there is a canonical $2$-morphism $d^2 \EquivTo 0$ given by properties of residue maps. Furthermore, we will show that each of the three components $(d^2)_i$ of $d^2$ can be identified with a shift of one of the iterated chiral brackets making up the Jacobi identity. For example, we define the composition
\begin{align*}
\mu_{Y, 1(23)}: j_{(3),*}j_{(3)}^*(\CB_Y \boxtimes \CB_Y \boxtimes \CB_Y) \to \Delta_{(3),!} \CB_Y
\end{align*}
as follows:
\begin{multline*}
 j_{(3),*}j_{(3)}^*(\CB_Y \boxtimes \CB_Y \boxtimes \CB_Y) \simeq j_{(1\ne 2,3),*}j_{(1\ne 2,3)}^* \left(\CB_Y \boxtimes j_*j^*(\CB_Y \boxtimes \CB_Y) \right)\\
 \to  j_{(1\ne 2,3),*}j_{(1\ne 2,3)}^* \left(\CB_Y \boxtimes \Delta_!(\CB_Y) \right)  \\
 \simeq  j_{(1\ne 2,3),*}j_{(1\ne 2,3)}^* \Delta_{(2=3),!}\left(\CB_Y \boxtimes \CB_Y \right)\\
 \simeq \Delta_{(2=3),!} j_{(3),*} j_{(3)}^* \left(\CB_Y \boxtimes \CB_Y \right) \\
 \to \Delta_{(2=3),!} \Delta_! \CB_Y\\ 
  \simeq  \Delta_{(3),!} \CB_Y.\\
\end{multline*}
\vskip -20 pt

The factorization isomorphism for the surjection $\id: \{1,2,3\} \surj \{1,2,3\}$ gives an identification
\begin{align*}
j_{(3),*}c_{\id,Y}:  j_{(3),*}j_{(3)}^* \CA_{Y^3} \EquivTo j_{(3),*}j_{(3)}^*(\CA_Y \boxtimes \CA_Y \boxtimes \CA_Y).
\end{align*}
The Ran isomorphism for the surjection $p: \{1,2,3\} \surj \{ \{1,2,3\}\} = \{ * \}$ gives an identification
\begin{align*}
\Delta_{(3),!}\nu_{p, Y}: \Delta_{(3),!} \Delta_{(3)}^! (\CA_{Y^3}[2]) \EquivTo \Delta_{(3),!}(\CA_Y[2]).
\end{align*}
Together, the factorization and Ran isomorphisms (and canonical base change isomorphisms) give two identifications of the sheaves $j_{(1\ne 2,3)}^* \Delta_{(2=3),!} \Delta _{(2=3)}^! \left(\CA_{Y^3}[1] \right)$ and $j_{(1\ne 2,3),*}j_{(1\ne 2,3)}^* \Delta_{(2=3),!}\left(\left(\CA_Y \boxtimes \CA_Y \right)[1]\right)$ as follows:
\begin{multline*}
f_1: j_{(1\ne 2,3),*} j_{(1\ne 2,3)}^* \Delta_{(2=3),!} \Delta _{(2=3)}^! \left(\CA_{Y^3}[1] \right) \simeq \Delta_{(2=3),!} \Delta_{(2=3)}^!  j_{(1\ne 2,3),*} j_{(1\ne 2,3)}^*\left(\CA_{Y^3}[1] \right)\\
\xrightarrow{\text{Factorization}} \Delta_{(2=3),!} \Delta_{(2=3)}^!  j_{(1\ne 2,3),*} j_{(1\ne 2,3)}^*\left(\left(\CA_{Y} \boxtimes \CA_{Y^2}\right) [1]\right)\\
\simeq j_{(1\ne 2,3),*} j_{(1\ne 2,3)}^*\left(\left(\CA_Y \boxtimes \Delta_! \Delta^! \CA_{Y^2}\right)[1] \right)\\
\xrightarrow{\text{Ran}}  j_{(1\ne 2,3),*} j_{(1\ne 2,3)}^*\left(\left(\CA_Y \boxtimes \Delta_!\CA_{Y}\right)[1] \right)\\
\simeq j_{(1\ne 2,3),*} j_{(1\ne 2,3)}^*\Delta_{(2=3),!}\left(\left(\CA_Y \boxtimes \CA_{Y}\right)[1] \right);\\
f_2: j_{(1\ne 2,3),*} j_{(1\ne 2,3)}^* \Delta_{(2=3),!} \Delta _{(2=3)}^! \left(\CA_{Y^3}[1] \right)
\xrightarrow{\text{Ran}} j_{(1\ne 2,3),*} j_{(1\ne 2,3)}^* \Delta_{(2=3),!} \left(\CA_{Y^2}[1] \right)\\
\simeq \Delta_{(2=3),!} j_* j^*\left(\CA_{Y^2}[1] \right) \\
\xrightarrow{\text{Factorization}}  \Delta_{(2=3),!} j_* j^*\left(\left(\CA_{Y} \boxtimes \CA_Y \right)[1] \right)\\
\simeq  j_{(1\ne 2,3),*} j_{(1\ne 2,3)}^*\Delta_{(2=3),!}\left(\left(\CA_Y \boxtimes \CA_{Y}\right)[1] \right).
\end{multline*}
Combining, we have the following diagram, in which the composition of the horizontal arrows on the top is $(d^2)_1$, one of the three components of $d^2$, and the composition of the horizontal arrows on the bottom is a shift of the iterated chiral bracket $\mu_{Y, 1(23)}$:
\begin{small}
\begin{center}
\begin{equation}\label{diag: Cousin complex}
\begin{tikzpicture}[baseline=(current  bounding  box.center),>=angle 90,  bij/.style={above,sloped,inner sep=0.5pt}]
\matrix(b)[matrix of math nodes, row sep=2em, column sep=3em, text height=1.5ex, text depth=0.25ex]
{ j_{(3),*}j_{(3)}^* \CA_{Y^3} & j_{(1\ne 2,3),*} j_{(1\ne 2,3)}^* \Delta_{(2=3),!} \Delta _{(2=3)}^! \left(\CA_{Y^3}[1] \right)  & \Delta_{(3),!} \Delta_{(3)}^! \left(\CA_{Y^3}[2] \right) \\
j_{(3),*}j_{(3)}^*(\CA_Y \boxtimes \CA_Y \boxtimes \CA_Y) & j_{(1\ne 2,3),*} j_{(1\ne 2,3)}^*\Delta_{(2=3),!}\left(\left(\CA_Y \boxtimes \CA_{Y}\right)[1] \right) &  \Delta_{(3),!}(\CA_Y[2]).\\};
\path[->, font=\scriptsize]
(b-1-1) edge node[above]{$\text{Res}$} (b-1-2)
(b-1-2) edge node[above]{$\text{Res}$} (b-1-3)
(b-2-1) edge (b-2-2)
(b-2-2) edge (b-2-3) 
(b-1-1) edge node[bij]{$\sim$} node[left]{$j_{(3),*}c_{\id,Y}$} (b-2-1)
(b-1-3) edge node[bij]{$\sim$} node[left]{$\Delta_{(3),!}\nu_{p, Y}$} (b-2-3)
(b-1-2) edge[bend left=30] node[right]{$f_2$} (b-2-2)
(b-1-2) edge[bend right=30] node[left]{$f_1$} (b-2-2);
\end{tikzpicture}
\end{equation}
\end{center}
\end{small}
Recalling the definition of $\mu_Y$ in terms of the factorization algebra $\{\CA_{Y^I}$, one can check that 2-morphism giving the compatibility between factorization isomorphisms makes the rightmost square commute. Likewise, the 2-morphism giving the compatibility between Ran isomorphisms makes the leftmost square commute. Finally, the compatibility isomorphism between Ran and factorization isomorphisms gives an identification $f_1 \EquivTo f_2$. 

All together, we obtain 2-morphisms
\begin{align*}
\mu_{Y, 1(23)} \EquivTo (d^2)_1[-3].
\end{align*}
We obtain a similar 2-morphism for each iterated chiral bracket and each component of $d^2$ (with appropriate signs); summing gives
\begin{align*}
\mu_{Y, 1(23)} - \mu_{Y, (12)3} - \mu_{Y,2(13)} \EquivTo d^2[-3] \EquivTo 0.
\end{align*}
This composition is the 2-morphism of the Jacobi identity for $\mu_Y$. The higher morphisms giving further compatibilities are obtained by similarly analyzing the Cousin complex for higher products of $Y$. 

With these preliminaries established, we can begin the comparison of $(\CB_X, \mu_X)$ and $(\CB^\prime_X, \mu_X^\prime)$. Recall that $\CB_X$ is defined by first taking the chiral algebra on $Y$ and then pulling back to $X$, whereas $\CB^\prime$ is defined by first pulling back the factorization algebra to $X$ and then taking the corresponding chiral algebra. 

First, on the level of $\CD_Y$-modules $\CB_X$ is by definition $\phi^* \CB_Y = \phi^* \CA_Y[-1]$. On the other hand, recalling the construction of the pullback of a factorization algebra, we see that $\CA_X = \phi^* \CA_Y$, and hence that $\CB^\prime_X = \phi^*\CA_Y[-1]$. 

To compare the brackets, it suffices to work over $V_\phi(2)$, since both brackets are trivial over $U_X$. There, $\mu_X$ is defined to be the pullback $j_V^*(\phi^2)^! \mu_Y$; recalling the definition of $\mu_X$ we see that this is simply the composition of the maps in the pullback of the diagram ($\ref{diagram: chiral bracket}$) along the map $\phi^2 \circ j_V$. But this is canonically isomorphic to the restriction $V_\phi(2)$ of the version of diagram ($\ref{diagram: chiral bracket}$) over $X$ defining the chiral bracket $\mu^\prime_X$ from the factorization algebra $\{\CA_{X^I}\}$. Thus we have a canonical isomorphism identifying $\mu_X$ with $\mu_X^\prime$.

The higher morphisms giving graded skew-symmetry of $\mu_X^\prime$ come from the graded skew-symmetry of $\tau_{\CA_X}$; recalling that $\CA_X = \phi^! \CA_Y$ it is straightforward to check that $\tau_{\CA_X}$ agrees with $(\phi^2)^! \tau_{\CA_Y}$, and the higher morphisms evincing the graded skew-symmetry of $\tau_{\CA_X}$ are the pullbacks of those for $\tau_{\CA_Y}$, which are exactly the higher morphisms we use to see the graded skew symmetry of $\mu_Y$. It follows that these pullbacks give the skew-symmetry structure both for $\mu_X^\prime$, and for the pullback $\mu_X$ of $\mu_Y$. 

Similarly, to compare the higher morphisms of the Jacobi identity, we first notice that it is enough to check over $V_\phi\{1,2,3\}$ (indeed, we have an open cover of $X^3$ given by $V_\phi\{1,2,3\}$ and the $U_{i \ne j,k}$, and for each $U_{i \ne j, k}$ the sum of the iterated brackets is canonically trivial). Denoting by $j_{V,3}$ the inclusion of $V_\phi\{1,2,3\}$ in $X^3$, we see that for both $\mu_X$ and $\mu_X^\prime$, the sum of the iterated brackets on $V_\phi\{1,2,3\}$ is trivialized by the morphisms making the pullback of the diagram (\ref{diag: Cousin complex}) along $\phi^2 \circ j_{V,3}$ commute. The same arguments also show that higher compatibility isomorphisms agree. As we saw in the discussion of weak factorization algebras, at a given number $n-1$ of iterations, there are only finitely many conditions to check on $X^n$, so proceeding in this manner gives a complete identification of the chiral algebras $(\CB_X , \mu_X)$ and $(\CB_X^\prime, \mu_X^\prime)$. 

\end{proof}

\bibliographystyle{alpha}
\bibliography{bibliography}

\end{document}